\documentclass[11pt]{amsart} \textwidth=14.5cm \oddsidemargin=1cm
\usepackage{amsmath,wasysym}
\usepackage{amsxtra}
\usepackage{amscd}
\usepackage{amsthm}
\usepackage{amsfonts}
\usepackage{amssymb}
\usepackage{eucal}
\usepackage{graphics, color}
\usepackage{hyperref,mathrsfs}
\usepackage[usenames,dvipsnames]{xcolor}
\usepackage{tikz}
\usepackage{stmaryrd}
\usepackage[all]{xy}
\usepackage{hyperref}
\usepackage{color}
\usepackage{bbm} 
\allowdisplaybreaks

\hypersetup{colorlinks,linkcolor=blue, urlcolor=blue,
citecolor=blue}

\newtheorem{thm}{Theorem} [section]
\newtheorem{lem}[thm]{Lemma}
\newtheorem{cor}[thm]{Corollary}
\newtheorem{prop}[thm]{Proposition}

\theoremstyle{definition}

\newtheorem{rem}[thm]{Remark}

\numberwithin{equation}{section}

\newcommand{\A}{\mathcal A}
\newcommand{\bu}{\mathbf i^{-}}

\newcommand{\B}{{\bf B}}

\newcommand{\C}{{\mathbb C}}
\newcommand{\D}{\mathcal{D}}

\newcommand{\End}{{\mathrm{End}}}
\newcommand{\F}{\mathbb{F}}
\newcommand{\ff}{\mathtt{f}}

\newcommand{\g}{{\mathfrak g}}

\newcommand{\HH}{\mathcal H}
\newcommand{\Hom}{{\mathrm{Hom}}}
\newcommand{\id}{\mathbbm{1}}

\newcommand{\K}{\mathbb{K}}

\newcommand{\mi}{\mathbf{i}}

\newcommand{\N}{{\mathbb N}}

\newcommand{\Q}{\mathbb {Q}}

\newcommand{\Sc}{\mathcal S}

\newcommand{\co}{\text{co}}
\newcommand{\ro}{\text{ro}}
\newcommand{\Tr}{\mathrm{Tr}}
\newcommand{\Z}{{\mathbb Z}}
\newcommand{\Mat}{\mathrm{Mat}}
\newcommand{\nc}{\newcommand}
\nc{\browntext}[1]{\textcolor{brown}{#1}}
\nc{\greentext}[1]{\textcolor{green}{#1}}
\nc{\redtext}[1]{\textcolor{red}{#1}}
\nc{\bluetext}[1]{\textcolor{blue}{#1}}
\nc{\brown}[1]{\browntext{ #1}}
\nc{\green}[1]{\greentext{ #1}}
\nc{\red}[1]{\redtext{ #1}}
\nc{\blue}[1]{\bluetext{ #1}}

\title[Asymptotic Schur algebras and cellularity of $q$-Schur algebras]
{Asymptotic Schur algebras and cellularity of $q$-Schur algebras}

\author[Weideng Cui]{Weideng Cui}
\address{School of Mathematics, Shandong University, Jinan, Shandong 250100, China}
\email{cwdeng@amss.ac.cn (Cui)}
\author[Li Luo]{Li Luo}
\author[Zheming Xu]{Zheming Xu}
\address{School of mathematical Sciences,
Shanghai Key Laboratory of Pure Mathematics and Mathematical Practice,
    East China Normal University, Shanghai 200241, China}
\email{lluo@math.ecnu.edu.cn (Luo), 51195500019@stu.ecnu.edu.cn (Xu)}


\begin{document}

\begin{abstract}
We prove that the $q$-Schur algebras of finite type introduced in \cite{LW22} are cellular in the sense of Graham and Lehrer, which is a generalization of Geck's theorem on the cellularity of Hecke algebras of finite type. Moreover, we study special modules of the associated asymptotic Schur algebras and left cell representations of Schur algebras, which generalize Lusztig's work about special modules of asymptotic Hecke algebras and left cell representations of Weyl groups, respectively.
\end{abstract}

\maketitle
\setcounter{tocdepth}{1}
\tableofcontents

\section{Introduction}

\subsection{}
The $q$-Schur algebra $\Sc_{n,d}$ of type $A$, introduced by Dipper and James \cite{DJ89}, is the endomorphism algebra of a direct sum of certain permutation modules (depending on $n$) of the Hecke algebra of type $A_{d-1}$. It also admits a geometric construction in terms of $n$-step partial flags in a $d$-dimensional space due to Beilinson, Lusztig and MacPherson \cite{BLM90}, which can trace back to Iwahori's realization of Hecke algebras (of type $A$) via complete flags (cf. \cite{I64}).

As a quotient of the general linear quantum group $\mathbf{U}_q(\mathfrak{gl}_n)$, the $q$-Schur algebra $\Sc_{n,d}$ is hidden in the Schur-Jimbo duality \cite{Jim86}, which demonstrates a double centralizer property between $\mathbf{U}_q(\mathfrak{gl}_n)$ and the Hecke algebra of type $A$ on certain Fock space. The Schur-Jimbo duality was
employed by Brundan in his remarkable work \cite{Br03}, where the Kazhdan-Lusztig theory was reformulated on the Fock space (instead of the Hecke algebra) to compute character formulas for the BGG catrgory $\mathcal{O}$ of general linear Lie superalgebras. Bao and Wang developed Brundan's approach to adapt the ortho-symplectic Lie superalgebras in \cite{BW18}, where a Schur-Jimbo duality between some $\imath$quantum groups and the Hecke algebras of type $B$ was established. It is worth mentioning that in the reformulation of the Kazhdan-Lusztig theory, the Fock space is regarded as the Grothendieck group of the BGG category $\mathcal{O}$.

\subsection{}
Motivated by the above mentioned connections to BGG category $\mathcal{O}$, the second author and Wang \cite{LW22} generalized the notion of $n$-step partial flags in terms of an arbitrary finite subset $X_\ff$ of weights invariant under the associated finite Weyl group $W$, and then introduced a new family of $q$-Schur algebras $\Sc_q(X_\ff)$ of finite type (see also affine type in \cite{CLW20}). Such a geometric definition helps the second and third authors \cite{LX22} derive a geometric Howe duality between $q$-Schur algebras.

For some special choices of $X_\ff$ in type $B$, it recovers the $q$-Schur algebras in \cite{Gr97,BKLW18,FL15} and proper subalgebras of the Schur algebras in \cite{DJM98,DS00}. Particularly, if we take $X_\ff$ to be a single regular $W$-orbit, then $\Sc_q(X_\ff)$ is just the Hecke algebra associated with $W$. In other words, $\Sc_q(X_\ff)$ are not only generalizations of various $q$-Schur algebras, but also of Hecke algebras. From this viewpoint, it is natural to ask whether various basic concepts and properties for Hecke algebras are still practicable for these $q$-Schur algebras $\Sc_q(X_\ff)$.

\subsection{}
In their significant work \cite{KL79}, Kazhdan and Lusztig introduced left, right and two-sided cells for an arbitrary Coxeter group through the so-called Kazhdan-Lusztig basis. Their definition of cells has the advantage that it provides not only representations of the Coxeter group, but also of the corresponding Hecke algebra. The theory of cells for Hecke algebras has been developed systematically by Lusztig \cite{Lu85, Lu87a, Lu87b, Lu87c, Lu89, Lu03} (see also \cite{Shi86, Xi94, BFO09}). It has close connections to primitive ideals, nilpotent orbits, representations of finite groups of Lie type, and so on. Many properties of cells for Hecke algebras are summarized in \cite{Lu03} and become known as Conjectures (P1)-(P15); these conjectures have been solved for finite/affine Hecke algebras of equal parameters among others, some of which require the positivity of Kazhdan-Lusztig bases.

A $q$-Schur algebra counterpart of the Kazhdan-Lusztig basis, called the canonical basis, has been constructed for $\Sc_q(X_\ff)$ in \cite{LW22}. Similar to Kazhdan-Lusztig bases, the canonical basis of $\Sc_q(X_\ff)$ also admits the positivity. Based on canonical bases and their positivity, the first two authors and Wang \cite{CLW20} formulated the $\mathfrak{a}$-functions, cells, Properties (P1)-(P15) and asymptotic forms for $\Sc_q(X_\ff)$ (called asymptotic Schur algebras), which are analogous to those for Hecke algebras established by Lusztig. In this paper, we shall provide several applications of the theory of cells and asymptotic Schur algebras, including the cellularity, special modules, blocks and left cell representations.

\subsection{}
Inspired by remarkable properties of the Kazhdan-Lusztig bases in Hecke algebras of type $A$, Graham and Lehrer \cite{GL96} introduced the notion of cellular algebras, which provides a systematic framework for studying the representation theory of non-semisimple algebras. Many classes of algebras have been shown to admit a cellular structure, such as Ariki-Koike algebras, cyclotomic $q$-Schur algebras, Doty's generalized $q$-Schur algebras, various classes of diagram algebras. In particular, Geck \cite{Ge07} showed that Hecke algebras of finite type are cellular by using deep properties of the Kazhdan-Lusztig bases and Lusztig's asymptotic Hecke algebras whenever the Hecke algebra is defined over a ground ring in which all bad primes are invertible.

In this paper, we shall show that the $q$-Schur algebra $\Sc_q(X_\ff)$ also admits a natural cellular structure using the properties of cells and asymptotic Schur algebras, which can be regarded as a generalization of Geck's theorem on the cellularity of Hecke algebras. This allows us to obtain a theory of Specht modules for $\Sc_q(X_\ff)$ via applying Graham and Lehrer's general approach.

\subsection{}
In \cite{Lu79}, Lusztig defined in a natural way a class of simple modules of $W$, which were later called special modules. These special modules exactly correspond to the special nilpotent orbits under the Springer correspondence, and they play a vital role in the classification of unipotent representations of a reductive group over a finite field \cite{Lu84}.

Lusztig \cite{Lu18} gave a new characterization of these special modules by a positivity property, when viewed them as modules of the corresponding asymptotic Hecke algebras. Inspired by this, we shall define and study special modules for the asymptotic Schur algebras.

In \cite{G96}, Gyoja showed that there are some connections between modular representations of Hecke algebras (more precisely, their partition into blocks) and families of characters and left cell representations. Later on, Rouquier \cite{Ro99} proved, depending on a description of the primite central idempotents in Lusztig's asymptotic Hecke algebras, that the families of characters of $W$ are exactly the blocks of the associated Hecke algebra over a suitable ring (called the Rouquier ring now).

In \cite[\S21-22]{Lu03}, Lusztig studied the left cell representations of $W$, and expected that they are precisely the so-called constructible representations. Lusztig \cite{Lu86} showed that this holds in the equal parameter case; for general weight functions, using ideas due to Gyoja and Rouquier, Geck \cite{Ge05} verified that this is true under the assumption that Conjectures (P1)-(P15) hold.

In this paper, we generalize the above results to the Schur algebra (the specialization of $\Sc_q(X_\ff)$ at $q=1$) and asymptotic Schur algebra. That is, we give a description of the primite central idempotents in asymptotic Schur algebras, and therefore of the blocks of $\Sc_q(X_\ff)$ over the Rouquier ring; besides, we establish some nice properties of left cell representations of the Schur algebras.



\subsection{}

The paper is organized as follows.

In Section~\ref{sec:conpro} we shall recall some basic constructions and properties of $\Sc_q(X_\ff)$ and the associated Schur algebras following \cite{LW22}. We also collect some results on cells for $\Sc_q(X_\ff)$ obtained in \cite{CLW20}. Moreover, we define a non-degenerate bilinear form on $\Sc_{q,\K}$; see \S\ref{sec:abiform}.

In Section~\ref{sec:repremj} we recall some properties of asymptotic Schur algebras obtained in \cite{CLW20}. Besides, we establish some results which will be used in Section~\ref{sec:cellularity}. For example, we study the Schur elements of $\Sc_{q,\K}$, and define the $\mathbf{a}$-invariants associated to irreducible representations of Schur algebras.

In Section~\ref{sec:cellularity} we shall show that the $q$-Schur algebra $\Sc_q(X_\ff)$ is a cellular algebra if all bad primes are invertible in the ground ring; see Theorem \ref{e5}.

In Section~\ref{sec:blocks} we shall define and study the special modules for asymptotic Schur algebras; see Theorem \ref{positive-lines}. Moreover, we shall study the blocks of $\Sc_q(X_\ff)$ and left cell representations of the Schur algebras; see Theorem \ref{g4}, Proposition \ref{g8}, Proposition \ref{h6}.


\vspace{2mm}
\noindent {\bf Acknowledgement.}
WC is partially supported by Young Scholars Program of Shandong University, Shandong Provincial Natural Science Foundation (No. ZR2021MA022) and the NSF of China (No. 11601273). LL is partially supported by the Science and Technology Commission of Shanghai Municipality (No. 21ZR1420000, 22DZ2229014) and Fundamental Research Funds for the Central Universities.

\section{$q$-Schur algebras and cells}
\label{sec:conpro}

In this section, we shall recall some basic constructions
and properties of $q$-Schur algebras of finite type and their cells, which were formulated in \cite{LW22} and \cite{CLW20}.
\subsection{Basic setting}
Let $\g$ be a complex simple Lie algebra of rank $d$ of arbitrary finite type, and $X$ its weight lattice.
%
The Weyl group $W$ of $\g$ is generated by the simple reflections
$S=\{s_1,s_2,\ldots,s_d\}$.
The length of $w\in W$ is denoted by $\ell(w)$. 

There is a natural right action of $W$ on $X$ defined by sending $\mi \in X \mapsto \mi w=w^{-1}(\mi)$. Take any $W$-invariant finite subset $X_\ff \subset X$ and denote
by $\Lambda_{\ff}$ the set of $W$-orbits in $X_{\ff}$.
Let
$\bu_\gamma$ be the unique anti-dominant element in a $W$-orbit $\gamma \in \Lambda_{\ff}$.

For any $\gamma\in\Lambda_{\ff}$, we define a subset of $\{1,2,\ldots,d\}$ as follows:
$J_\gamma:=\{k~|~ 1\leq k\leq d, \bu_\gamma s_k=\bu_\gamma\}$.
%
Denote by $W_\gamma$ the parabolic subgroup of $W$ generated by $\{ s_j \mid j \in J_\gamma\}$.
Let
 $\D_{\gamma}$ (resp. $\D_{\gamma}^+$) be the set of the minimal (resp. longest) length right coset representatives of $W_\gamma$ in $W$, and let $\D_{\gamma\nu}$ (resp. $\D_{\gamma\nu}^+$) be the set of the minimal (resp. longest) length double coset representatives for $W_\gamma\setminus W/W_\nu$, ($\gamma,\nu\in\Lambda_{\ff}$).

\subsection{$q$-Schur algebras}
\label{sec:sch}
%
Let $q$ be an indeterminate and denote
\[
\A:=\Z[q,q^{-1}] \quad\mbox{and}\quad \K:=\mathrm{Frac}(\A)=\Q(q).
\]
The \emph{Hecke algebra} $\HH$ associated to $W$ is an $\A$-algebra with identity $\id$ generated by $H_1,H_2,\ldots,H_d$ with relations
\begin{equation*}
(H_i-q^{-1})(H_i+q)=0;\quad \underbrace{H_iH_jH_i\cdots}_{m_{ij}}=\underbrace{H_jH_iH_j\cdots}_{m_{ij}}, \quad (1\leq i\neq j\leq d),
\end{equation*}
where $m_{ij}$ is the order of $s_is_j$ in $W$.
Write $H_w=H_{i_1}H_{i_2}\cdots H_{i_l}$ if $w=s_{i_1}s_{i_2}\cdots s_{i_l}\in W$ is a \emph{reduced} word.

For $\gamma \in \Lambda_{\ff}$, we define the \emph{$q$-symmetrizer} $x_\gamma \in\HH$ by
\begin{align*}
x_\gamma  :=\sum_{w\in W_\gamma} q^{\ell(w_\circ^\gamma)-\ell(w)}H_w,
\end{align*}
where $w_\circ^\gamma$ is the unique longest element in $W_\gamma$.

We define the \emph{$q$-Schur algebra} $\Sc_{q}$ over $\A$ as follows:
\begin{equation*}
\Sc_{q}=\Sc_{q}(X_\ff):=\End_{\HH}(\bigoplus_{\gamma\in\Lambda_{\ff}}x_\gamma\HH)=\bigoplus_{\gamma,\nu\in\Lambda_{\ff}}\Hom_{\HH}(x_\nu\HH,x_\gamma\HH),
\end{equation*}
which depends on the choices of $W$ and $X_\ff$. In particular, the Hecke algebra can be regarded as a special $q$-Schur algebra, for which we take $\Lambda_{\ff}$ any single regular $W$-orbit.

\vspace{0.3cm}
\noindent{\em Convention:}
For any commutative ring $R$ with a ring homomorphism $\A \rightarrow R$, we shall often add a subscript $R$ on the bottom-right of an $\A$-space (or an $\A$-map) to mean the base change $R \otimes_{\A}-$, e.g. $\HH_\K$, $\Sc_{q,\K}$ by taking $R=\K$.
\vspace{0.3cm}

We have
$$\Sc_{q,\K} \cong \End_{\HH_\K}(\bigoplus_{\gamma\in\Lambda_{\ff}}x_\gamma\HH_\K).$$
It is a split semisimple algebra because $\HH_\K$ is split semisimple (cf. \cite[Theorem~9.3.5]{GP00}).

The classical limit of $\Sc_q$ (resp. $\Sc_{q,\K}$) is denoted by $\Sc$ (resp. $\Sc_\Q$). They are called \emph{Schur algebras}. In fact, $\Sc=\Sc_{q,\Z}$ and $\Sc_\Q=\Sc_{q,\Q}$, where we regard $\Z$ and $\Q$ as $\A$-algebras via $q\mapsto1$.
Since the group algebra $\Q[W]$ (i.e. the classical limit of $\HH_\K$) is split semisimple by \cite[Theorem~6.3.8]{GP00}, it is not hard to show that $\Sc_\Q$ is split semisimple, too.


\subsection{Bases}\label{bases}
Denote
\begin{equation*}   \label{eq:Xi}
\Xi:=\{C=(\gamma,g,\nu)~|~\gamma,\nu\in\Lambda_{\ff}, g\in\D_{\gamma\nu}\},
\end{equation*}
on which we can introduce a partial order $<$ by
 $$(\gamma,g,\nu)<(\gamma',g',\nu') \quad \Leftrightarrow\quad \gamma=\gamma', \nu=\nu', g<g'.$$ Here $g<g'$ uses the Bruhat order on $W$.

For any $C=(\gamma,g,\nu)\in \Xi$, we sometimes write
\begin{equation*}
\ro(C) =\gamma, \quad \co(C)=\nu, \quad w_C^+=\mbox{the unique longest element in } W_\gamma gW_\nu.
\end{equation*} 
Moreover, we denote $$H_C:=\sum_{w\in W_\gamma gW_\nu}q^{-\ell(w)}H_w.$$

For $C=(\gamma,g,\nu)\in \Xi$, let $\phi_{\gamma\nu}^g\in \Sc_{q}$ be the element such that
\begin{align}  \label{eq:phig}
\phi_{\gamma\nu}^g(x_{\nu'})=\delta_{\nu,\nu'} q^{\ell(w_\circ^\nu)}H_C, \quad \forall \nu' \in \Lambda_{\ff}.
\end{align}
It is shown in \cite[Proposition 3.5]{LW22} that $\{\phi_{\gamma\nu}^g ~|~(\gamma,g,\nu)\in\Xi\}$ forms an $\A$-basis of $\Sc_{q}$. Denote
\begin{equation}\label{stbasis}
[C]:=q^{\ell(w_C^+)-\ell(w^\nu_\circ)}\phi_{\gamma\nu}^g\quad\mbox{for $C=(\gamma,g,\nu)\in \Xi$}.
\end{equation}
The set $\{[C] ~|~ C\in \Xi\}$ is called a \emph{standard basis} of $\Sc_{q}$.
Moreover, there exists a \emph{canonical basis} $\B (\Sc_{q})=\big\{\{C\} ~|~ C \in \Xi \big\}$ for $\Sc_{q}$ which satisfies that
\begin{equation}\label{eq:can}
\{C\}=[C]+\sum_{C'<C} P_{C',C}[C']\quad \mbox{with $P_{C',C} \in q\N[q]$};
\end{equation} refer to \cite[\S3.5]{LW22} for more details. Especially, in the case of Hecke algebras, the canonical basis is the well-known \emph{Kazhdan-Lusztig basis} $\{\mathcal{C}_w ~|~ w \in W\}$. 
By \cite[Lemma~2.5 \& Proposition~3.4]{CLW20}, we have
\begin{equation}\label{KLP}
	\mathcal{C}_{w_C^+}= \sum_{C'\leq C} p_{w_{C'}^+,w_C^+} q^{\ell(w_{C'}^+)} H_{C'}, \quad \mbox{and} \quad P_{C',C}=p_{w_{C'}^+,w_C^+}, \quad \mbox{for} \quad C'\leq C
\end{equation}
where $p_{w_{C'}^+,w_C^+}$ are the Kazhdan-Lusztig polynomials.

Let $g_{A,B}^{C}\in \A$ (resp. $h_{x,y}^{z}\in \A$) be the structure constants of $\Sc_q$ (resp. $\HH$) with respect to the canonical basis (resp. the Kazhdan-Lusztig basis). That is,
\begin{equation*}\label{a3}
\{A\}\{B\}=\sum_{C\in \Xi} g_{A,B}^{C} \{C\} \quad \mbox{(resp. $\mathcal{C}_x \mathcal{C}_y=\sum_{z \in W} h_{x,y}^{z}\mathcal{C}_z)$}.
\end{equation*}
The positivity property derived from \cite[Theorem 4.5]{LW22} says that $$g_{A,B}^{C}, h_{x,y}^z \in \N[q,q^{-1}].$$

\subsection{The anti-automorphism $\Psi$}
For any $C=(\gamma,g,\nu)\in \Xi$, we shall write \begin{equation}
\label{eq:Ct}
C^t=(\nu,g^{-1},\gamma)\in\Xi.
\end{equation}
Define an $\A$-module homomorphism
\begin{equation}
\label{eq:psi}
\Psi:\Sc_{q}\rightarrow \Sc_{q}, \quad [C]\mapsto[C^t] \quad(\forall C\in\Xi).
\end{equation}
It has been shown in \cite[Lemma~3.6]{CLW20} that $\Psi: \Sc_{q}\rightarrow \Sc_{q}$
is an anti-automorphism of $\Sc_{q}$ and $\Psi(\{C\})=\{C^t\}$ $(\forall C\in\Xi)$.


\subsection{A bilinear form on $\Sc_{q,\K}$}
\label{sec:abiform}
Let $\mathcal{M}_{\ff}$ denote $\Z(\Lambda_{\ff} \times \Lambda_{\ff})$ the set of $\Z$-valued functions on $\Lambda_{\ff} \times \Lambda_{\ff}$. It is a $\Z$-algebra endowed with the natural convolution product and $\{e_{\gamma\nu}|(\gamma, \nu)\in\Lambda_{\ff}\times\Lambda_{\ff}\}$ forms a $\Z$-basis of $\mathcal{M}_{\ff}$, where $e_{\gamma\nu}$ is the corresponding characteristic function. Set $$\Sc_{q}':=\HH \otimes_\Z \mathcal{M}_{\ff},$$
 which is an $\A$-algebra under the obvious product. Similarly, we can introduce $\K$-algebras $\mathcal{M}_{\ff,\K}=\mathcal{M}_{\ff}\otimes_{\Z}\K=\K(\Lambda_{\ff} \times \Lambda_{\ff})$ and $\Sc_{q,\K}'=\HH_\K \otimes_\K \mathcal{M}_{\ff,\K}$.

There exists a non-degenerate symmetric $\A$ (resp. $\Z$)-bilinear form $(\cdot, \cdot)_{1}$ (resp. $(\cdot, \cdot)_{2}$) on $\HH$ (resp. $\mathcal{M}_{\ff}$) defined by $(H_w, H_{w'})_{1}=\delta_{w,w'^{-1}}$ (resp. $(e_{\gamma\nu},e_{\gamma'\nu'})_{2}=\delta_{(\gamma,\nu), (\nu',\gamma')}$). Thus, we can define a non-degenerate symmetric $\A$-bilinear form $(\cdot, \cdot)_{3}$ on $\Sc_{q}'$ by $(a \otimes b, a' \otimes b')_{3}=(a,a')_{1}\cdot (b,b')_{2}$ for any $a, a'\in \HH$ and $b, b'\in \mathcal{M}_{\ff}$.

For a subset $Y \subset W$, we denote the associated Poincar\'{e} polynomial by $$\mathfrak{P}_Y:=\sum_{w \in Y} q^{-2\ell(w)}.$$
Let
\begin{equation*}\label{mf}
	m_\gamma:=\frac{q^{-\ell(w_\circ^\gamma)}}{\mathfrak{P}_{W_\gamma}}x_\gamma \otimes e_{\gamma\gamma} \quad \mbox{and} \quad m_\ff:=\sum_{\gamma \in \Lambda_{\ff}} m_\gamma \in \Sc_{q,\K}'.
\end{equation*}
Since $e_{\gamma\nu}e_{\gamma'\nu'}=\delta_{\nu\gamma'}e_{\gamma\nu'}$ and $x_\gamma^2=q^{\ell(w_\circ^\gamma)} \mathfrak{P}_{W_\gamma} x_\gamma$, we see that $m_\gamma$'s are idempotent and mutually orthogonal. Hence $m_\ff$ is idempotent, too.

It is easy to check that $(\bigoplus_{\gamma \in \Lambda} x_\gamma \HH_\K)\otimes_\K \K\Lambda_{\ff} \cong m_\ff \Sc_{q,\K}'$ as right $\Sc_{q,\K}'$-modules,
and
\begin{align*}
	\End_{\Sc_{q,\K}'}\big((\bigoplus_{\gamma \in \Lambda} x_\gamma \HH_\K\big)\otimes_\K \K\Lambda_{\ff}) &\cong \End_{\HH_\K}(\bigoplus_{\gamma \in \Lambda} x_\gamma \HH_\K) \otimes_\K \End_{\mathcal{M}_{\ff,\K}}(\K \Lambda_{\ff})\\
	&\cong \End_{\HH_\K}(\bigoplus_{\gamma \in \Lambda} x_\gamma \HH_\K)=\Sc_{q,\K}.
\end{align*}
 So we have a canonical $\K$-algebra isomorphism $\widetilde{\quad}: \Sc_{q,\K} \rightarrow \End_{\Sc_{q,\K}'}(m_\ff\Sc_{q,\K}')$ via \\$\widetilde{\phi_{\gamma\nu}^g}(m_\mu)=\delta_{\mu\nu} \frac{q^{\ell(w_\circ^\nu)-\ell(w_\circ^\gamma)}}{\mathfrak{P}_{W_\gamma}}H_C \otimes e_{\gamma\nu}$, which induces a $\K$-algebra isomorphism
$$\varphi: \Sc_{q,\K} \rightarrow m_\ff\Sc_{q,\K}'m_\ff,\quad \varphi(\phi_{\gamma\nu}^g)=\frac{q^{\ell(w_\circ^\nu)-\ell(w_\circ^\gamma)}}{\mathfrak{P}_{W_\gamma}}H_C \otimes e_{\gamma\nu}.$$
Particularly, $\varphi(\{C\})=\frac{q^{-\ell(w_\circ^\gamma)}}{\mathfrak{P}_{W_\gamma}}\mathcal{C}_{w_C^+} \otimes e_{\gamma\nu}$ by \eqref{KLP}.

\begin{thm}\label{blf}
	There exists a symmetric associative $\K$-bilinear form $(\cdot, \cdot)$ on $\Sc_{q,\K}$ by $$(x,y):=(\varphi(x), \varphi(y))_{3} \quad\mbox{for any } x,y \in \Sc_{q,\K}.$$
	Precisely, For any $A, B\in \Xi$ with $A=(\gamma, g, \nu)$,
	\begin{equation}
		\label{eq:bi}
		([A], [B])=\delta_{A,B^t} \frac{q^{2\ell(w_A^+)-\ell(w_\circ^\gamma)-\ell(w_\circ^\nu)}}{\mathfrak{P}_{W_\gamma \cap g W_\nu g^{-1}}},
	\end{equation}
	which implies that $(\cdot, \cdot)$ is non-degenerate.
\end{thm}
\begin{proof}
	Recall $[A]=q^{\ell(w_A^+)-\ell(w^\nu_\circ)}\phi_{\gamma\nu}^g$ from \eqref{stbasis}. The formula follows from a direct calculation below:
	\begin{align*}
		(\phi_{\gamma\nu}^g,\phi_{\gamma'\nu'}^{g'})&=\delta_{(\gamma,\nu), (\nu',\gamma')}(\frac{q^{\ell(w_\circ^\nu)-\ell(w_\circ^\gamma)}}{\mathfrak{P}_{W_\gamma}}H_C, \frac{q^{\ell(w_\circ^\gamma)-\ell(w_\circ^\nu)}}{\mathfrak{P}_{W_\nu}}H_{C'})\\
		&=\delta_{(\gamma,g,\nu), (\nu',g'^{-1},\gamma')}\frac{\mathfrak{P}_{W_\gamma g W_\nu}}{\mathfrak{P}_{W_\gamma} \mathfrak{P}_{W_\nu}}
		=\delta_{(\gamma,g,\nu), (\nu',g'^{-1},\gamma')}\frac{1}{\mathfrak{P}_{W_\gamma \cap g W_\nu g^{-1}}}.
	\end{align*}
\end{proof}

\begin{rem}
In \cite[\S4.1]{CLW20}, we introduced another symmetric bilinear form on $\Sc_{q,\K}$, which is not associative (cf. \cite[Lemma 4.2]{CLW20}) and hence is different from the bilinear form $(\cdot, \cdot)$ here. That bilinear form admits some nice properties (cf. \cite[Theorem 4.4]{CLW20}). We do not expect that the bilinear form $(\cdot, \cdot)$ has similar properties.
\end{rem}
%

\subsection{The $\mathfrak{a}$-function}

For $A,B,C \in \Xi$, let $n_{A, B}^{C}$ be the minimal non-negative integer $n$ such that $q^{n}g_{A, B}^{C}\in \mathbb{Z}[q]$. We define the \emph{$\mathfrak{a}$-function} $\mathfrak{a}: \Xi \rightarrow \mathbb{N}$ as follows:
\begin{equation}\label{def:afunction}
\mathfrak{a}(C)=\max\left\{n_{A, B}^{C}+\ell(w_\circ^{\co(A)}) ~|~ A, B \in \Xi \right\},
\end{equation} which can be viewed as a $q$-Schur algebra analog of Lusztig's $\mathbf{a}$-function $\mathbf{a}:W\rightarrow\N$.

It follows from the definition of the $\mathfrak{a}$-function that $q^{\mathfrak{a}(C)-\ell(w_\circ^{\co(A)})} g_{A, B}^{C} \in \Z[q]$. Recall $C^t$ in \eqref{eq:Ct}. We set
\[
\gamma_{A, B}^{C^{t}} \in \Z \text{ such that } q^{\mathfrak{a}(C)-\ell(w_\circ^{\co(A)})} g_{A, B}^{C} -\gamma_{A, B}^{C^{t}} \in q\Z[q].
\]
Clearly, for each $C\in\Xi$, there exist $A,B\in\Xi$ such that $\gamma_{A, B}^{C^{t}} \neq 0$.

\subsection{Distinguished elements}
For $C \in \Xi$ with $\ro(C)=\co(C)=\gamma$, if $P_{w_\circ^\gamma, w_C^+} \neq 0$, we define $\Delta(C) \in \N$ and $n_C \in \Z^+$ to be the unique integers such that
\begin{align*}
\label{eq:nC}
p_{w_\circ^\gamma, w_C^+} \in n_Cq^{-\Delta(C)}+q^{-\Delta(C)+1}\N[q].
\end{align*}

\begin{rem}
  Note that the $\mathfrak{a}$-function defined in \cite{CLW20} is not the $\mathfrak{a}(C)$ defined in \eqref{def:afunction}, but $\mathfrak{a}(C)-\ell(w_\circ^{\co(C)})$. It can also be checked directly that the $\Delta$-function defined therein is $\Delta(C)-\ell(w_\circ^{\co(C)})$ (i.e. by the same shift as $\mathfrak{a}$-function) and $n_C$ is the same. We shall revisit some results obtained in \cite{CLW20} in the sequel to adapt these changes and the proofs will be omitted since they are almost the same as those in \cite{CLW20}.
\end{rem}

An element $D\in\Xi$ is called {\em distinguished} if $\ro(D) =\co(D)$ and $\mathfrak{a}(D)=\Delta(D).$ We denote by $\mathcal{D}$ the set of distinguished elements in $\Xi$.



\subsection{Cells}
For $C, C'\in \Xi$, we write $C\preceq_{L}C'$ (resp. $C\preceq_{R}C'$ and $C\preceq_{LR}C'$) if the canonical basis element $\{C\}$ appears with nonzero coefficient in the canonical basis expansion of $h\{C'\}$ (resp. $\{C'\}h'$ and $h\{C'\}h'$), for some $h,h'\in \B (\Sc_{q})$. These relations are pre-orders. They induce the corresponding equivalence relations $\sim_{L}$, $\sim_{R}$ and $\sim_{LR}$, and the resulting equivalence classes are called the left, right and two-sided cells of $\Xi$, respectively. Alternatively, a cell can be understood as consisting of the canonical basis elements $\{C\}$ in $\Sc_{q}$ parameterized by elements $C$ in a cell of $\Xi$.

The same notions and notations above were defined for the Weyl group $W$ or $\HH$ with respect to the Kazhdan-Lusztig basis (cf. \cite{KL79, Lu03}).

\begin{prop}[{\cite[Proposition 5.6]{CLW20}}]
\label{prop:d=d2}
Let $C, C'\in \Xi$. Then we have
\begin{enumerate}
\item
$C\preceq_{L}C'$ if and only if $\co(C)=\co(C')$ and $w_{C}^{+}\preceq_{L} w_{C'}^{+}$. Similarly, $C\preceq_{R}C'$ if and only if $\ro(C)=\ro(C')$ and $w_{C}^{+}\preceq_{R} w_{C'}^{+}$.
\item
$C\sim_{L}C'$ if and only if $\co(C)=\co(C')$ and $w_{C}^{+}\sim_{L} w_{C'}^{+}$. Similarly, $C\sim_{R}C'$ if and only if $\ro(C)=\ro(C')$ and $w_{C}^{+}\sim_{R} w_{C'}^{+}$.
\item
$C\sim_{LR}C'$ if and only if $w_{C}^{+}\sim_{LR} w_{C'}^{+}$.
\item
$\mathfrak{a}(C)=\mathbf{a}(w_{C}^{+})$, where $\mathbf{a}$ is Lusztig's $\mathbf{a}$-function on $W$.
\end{enumerate}
\end{prop}

Then we have the following corollary.
\begin{cor}\label{lem:d=d3}
\begin{enumerate}
\item
For any $C\in \Xi$, we have $\mathfrak{a}(C)=\mathfrak{a}(C^{t})$.
\item
For any $A, B, C\in \Xi$, we have $\gamma_{A, B}^{C}=\gamma_{B^{t}, A^{t}}^{C^{t}}$.
\end{enumerate}
\end{cor}
\begin{proof}
(1) It follows from Proposition \ref{prop:d=d2}(4) together with \cite[Proposition 13.9(a)]{Lu03} which says $\mathbf{a}(w)=\mathbf{a}(w^{-1})$ for any $w\in W$.

(2) Applying the anti-automorphism $\Psi$ defined in \eqref{eq:psi}, we have $g_{A, B}^{C}=g_{B^{t}, A^{t}}^{C^{t}}$. That is,
\begin{align*}
q^{-\mathfrak{a}(C)+\ell(w_\circ^\mu)}\gamma_{A, B}^{C^{t}}+\mathrm{higher~terms}=q^{-\mathfrak{a}(C^{t})+\ell(w_\circ^\mu)}\gamma_{B^{t}, A^{t}}^{C}+\mathrm{higher~terms},
\end{align*}
where $\ro(B)=\mu=\co(A)$. By (1), we must have $\gamma_{A, B}^{C^{t}}=\gamma_{B^{t}, A^{t}}^{C}$.
\end{proof}

\subsection{Properties of cells}
The following theorem is a $q$-Schur algebra analog of properties for cells in Hecke algebras (cf. \cite[Conj.~ 14.2(P1)-(P15)]{Lu03}, which holds for finite/affine Hecke algebras of equal parameters \cite[\S15]{Lu03}). (Note that a counterpart is not formulated here for (P12) of \cite[Conj.~ 14.2]{Lu03} which compares the $\mathbf{a}$-functions on a Coxeter group and its parabolic subgroup.)

\begin{thm}[{\cite[Theorem~5.8]{CLW20}}]
\label{thm:cell structures}
The following properties hold.
\begin{itemize}
\item[(P1)] For any $C \in \Xi$ with $\ro(C)=\co(C)$, we have $\mathfrak{a}(C)\leq\Delta(C)$.
\item[(P2)] If $D\in \mathcal{D}$ and $\gamma_{A, B}^{D}\neq 0$, for $A, B \in \Xi$, then $B=A^{t}$.
\item[(P3)] For each $C\in \Xi$, there is a unique $D\in \mathcal{D}$ such that $\gamma_{C, C^{t}}^{D}\neq 0$.
\item[(P4)] If $C\preceq_{LR}C'$, then $\mathfrak{a}(C)\geq \mathfrak{a}(C')$. Hence, if $C\sim_{LR}C'$, then $\mathfrak{a}(C)=\mathfrak{a}(C')$.
\item[(P5)] If $D\in \mathcal{D}$ and $\gamma_{A^{t}, A}^{D}\neq 0$ for $A\in \Xi$, then $\gamma_{A^{t}, A}^{D}=n_{D}=1$.
\item[(P6)] If $D\in \mathcal{D}$, then $D=D^{t}$.
\item[(P7)] We have $\gamma_{A, B}^{C}=\gamma_{C, A}^{B}=\gamma_{B, C}^{A}$, for any $A, B, C \in \Xi$.
\item[(P8)] Suppose that $\gamma_{A, B}^{C}\neq 0$, for $A, B, C \in \Xi$. Then $A\sim_{L}B^{t}$, $B\sim_{L}C^{t}$ and $C\sim_{L}A^{t}$.
\item[(P9)] If $C\preceq_{L}C'$ and $\mathfrak{a}(C)=\mathfrak{a}(C')$, then $C\sim_{L}C'$.
\item[(P10)] If $C\preceq_{R}C'$ and $\mathfrak{a}(C)=\mathfrak{a}(C')$, then $C\sim_{R}C'$.
\item[(P11)] If $C\preceq_{LR}C'$ and $\mathfrak{a}(C)=\mathfrak{a}(C')$, then $C\sim_{LR}C'$.
%
\item[(P13)] Each left cell $\Gamma$ in $\Sc_q$ contains a unique distinguished element $D$. Moreover, we have $\gamma_{C^{t}, C}^{D}\neq 0$ for all $C\in \Gamma$.
\item[(P14)] We have $C\sim_{LR}C^{t}$, for any $C\in \Xi$.
\item[(P15)] Let $q'$ be a second indeterminate and let $\hat{g}_{A, B}^{C}\in \mathbb{Z}[q', q'^{-1}]$ be obtained from $g_{A, B}^{C}$ by the substitution $q\mapsto q'$. If $C, C', A, B\in \Xi$ satisfy $\mathfrak{a}(A)=\mathfrak{a}(B)$, then
\[
\sum\limits_{A''\in \Xi}\hat{g}_{B, C'}^{A''}g_{C, A''}^{A}=\sum\limits_{A'\in \Xi}g_{C, B}^{A'}\hat{g}_{A', C'}^{A}.
\]
\end{itemize}
\end{thm}

\section{Asymptotic Schur algebras}
\label{sec:repremj}

\subsection{Asymptotic Schur algebras}
\label{sec:asymp}
As the same as in \cite[\S6.1]{CLW20}, we can define a $\mathbb{Z}$-algebra $\mathbf{J}$ named the \emph{asymptotic Schur algebra}. This algebra has a $\mathbb{Z}$-basis $\{t_{C}\:|\:C\}$ with the multiplication
\[
t_{A}t_{B}=\sum\limits_{C\in \Xi}\gamma_{A,B}^{C^{t}}t_{C},
\]
and the identity $\sum_{D\in \mathcal{D}}t_{D}$. Moreover, the set $\{t_D ~|~ D \in \mathcal{D}\}$ is a set of pairwise orthogonal idempotents and $t_C t_D=t_C$ if $C \sim_{L} D$ and $0$ otherwise, $t_D t_C=t_C$ if $C \sim_{R} D$ and $0$ otherwise.

For a subset $Y \subset \Xi$, we set $\mathbf{J}_Y=\sum_{C \in Y} \Z t_C$. Let $B_0$ be the set of all two-sided cells in $\Sc_{q}$ and $\mathcal{D}_{\mathbf{c}}=\mathcal{D}\cap\mathbf{c}$. Then $\mathbf{J}_{\mathbf{c}}$ is an ideal of $\mathbf{J}$, moreover, an algebra with the identity $t_{\mathbf{c}}:=\sum_{D\in \mathcal{D}_{\mathbf{c}}}t_{D}$ and
$$\mathbf{J}=\bigoplus_{\mathbf{c}\in B_{0}}\mathbf{J}_{\mathbf{c}}.$$
Let $\tau : \mathbf{J} \rightarrow \Z$ be the $\Z$-map such that $\tau(t_{C})=1$ if $C\in \mathcal{D}$ and $0$ otherwise. Then clearly the bilinear form defined by $(t_A,t_B)=\tau(t_A t_B)=\delta_{A,B^t}$ is non-degenerate symmetric associative.

The following proposition characterizes the left, right and two-sided cells in $\Sc_{q}$.
\begin{prop}[{\cite[Proposition~6.4]{CLW20}}]
\label{characterization-two-cells}
Let $C, C'\in \Xi$.
\begin{itemize}
\item[(1)] $C\sim_{L}C'$ if and only if $t_{C}t_{C'^{t}}\neq 0$; $C\sim_{L}C'$ if and only if $t_{C'}$ appears with nonzero coefficient in $t_{A}t_{C}$ for some $A\in \Xi$.
\item[(2)] $C\sim_{R}C'$ if and only if $t_{C^{t}}t_{C'}\neq 0$; $C\sim_{R}C'$ if and only if $t_{C'}$ appears with nonzero coefficient in $t_{C}t_{A}$ for some $A\in \Xi$.
\item[(3)] $C\sim_{LR}C'$ if and only if $t_{C}t_{A}t_{C'}\neq 0$ for some $A\in \Xi$; $C\sim_{LR}C'$ if and only if $t_{C'}$ appears with nonzero coefficient in $t_{A'}t_{C}t_{A}$ for some $A', A\in \Xi$.
\end{itemize}
\end{prop}

\subsection{The homomorphism $\Phi$}
\label{sec:phipro}
In \cite[Theorem 6.5 \& Remark 6.6]{CLW20}, an $\A$-algebra homomorphism $\Phi :\Sc_{q}\rightarrow \mathbf{J}_\A$, which preserves the identity elements, is defined by
\begin{align*}
\Phi (\{C\})=\sum_{\mathbf{c}\in B_{0}, D\in \mathcal{D}_{\mathbf{c}}, A \in \mathbf{c}}g_{C, D}^{A}t_{A}=\sum_{D\in \mathcal{D}, A \in \Xi, \mathfrak{a}(D)=\mathfrak{a}(A)}g_{C, D}^{A}t_{A}.
\end{align*}
It is a $q$-Schur algebra analog of the one for Hecke algebras established in \cite[\S2.4]{Lu87a} (see \cite[Theorem 2.3]{Du95} and \cite[(4.9)]{Mc03} for finite and affine type $A$, respectively).

If we identify the $\A$-modules $\Sc_{q}$ and $\mathbf{J}_\A$ via $\{C\}\leftrightarrow t_{C}$, then the regular $\mathbf{J}_\A$-action on $\mathbf{J}_\A$ implies the left $\mathbf{J}_\A$-module structure on $\Sc_{q}$ given by
\begin{align*}
	t_{A} \ast \{B\}=\sum\limits_{C\in \Xi}\gamma_{A,B}^{C^{t}}\{C\}.
\end{align*}
For any $a\in \mathbb{N}$, we set
\begin{align*}
	\Sc_{q}^{\geq a}:=\sum_{C \in \Xi; \mathfrak{a}(C)\geq a}\A\{C\}.
\end{align*}
It is easy to check that $\Sc_{q}^{\geq a}$ is an ideal of $\Sc_{q}$ and $\mathbf{J}_\A \ast \Sc_{q}^{\geq a} =\Sc_{q}^{\geq a}$ for any $a \in \N$.
\begin{lem}\label{b2}
For any $s, s'\in \Sc_{q}$, $B \in \Xi$, we have
\begin{align*}
	s\cdot \{B\} \equiv \Phi(s) \ast \{B\} \quad \mathrm{mod~} \Sc_{q}^{\geq \mathfrak{a}(B)+1}.
\end{align*}
\end{lem}
\begin{proof}
It suffices to show
\begin{align*}
	\{C\}\cdot \{B\}\equiv \Phi(\{C\}) \ast \{B\}\quad \mathrm{mod~} \Sc_{q}^{\geq \mathfrak{a}(B)+1}.
\end{align*}
By Theorem \ref{thm:cell structures}(P15), similar to the proof of \cite[\S18.8(b)]{Lu03}, if $\mathfrak{a}(A)=\mathfrak{a}(B)$ we have
\begin{align}\label{b3}
	\sum\limits_{A'\in \Xi}g_{C, B}^{A'}\gamma_{A', C'}^{A^{t}}=\sum\limits_{A''\in \Xi}g_{C, A''}^{A}\gamma_{B, C'}^{A''^{t}}.
\end{align}
Thus we compute directly that
\begin{align*}
	\Phi(\{C\})\ast \{B\}=&\sum_{D\in \mathcal{D}, A \in \Xi, \mathfrak{a}(D)=\mathfrak{a}(A)}g_{C, D}^{A}t_{A}\ast \{B\}
	=\sum_{\substack{D\in \mathcal{D}, A,E \in \Xi \\ \mathfrak{a}(D)=\mathfrak{a}(A)}} g_{C, D}^{A}\gamma_{A,B}^{E^{t}}\{E\}\\
	=&\sum_{\substack{D \in \mathcal{D}, A,E \in \Xi \\ \mathfrak{a}(D)=\mathfrak{a}(E)}} g_{C, D}^{A}\gamma_{A,B}^{E^{t}}\{E\} \qquad \text{ by Theorem~\ref{thm:cell structures} (P4, P8)}\\
	=&\sum_{\substack{D\in \mathcal{D}, A, E \in \Xi \\ \mathfrak{a}(D)=\mathfrak{a}(E)}} g_{C, A}^{E}\gamma_{D,B}^{A^{t}}\{E\}   \qquad \text{ by \eqref{b3}}\\
	=&\sum_{E \in \Xi, \mathfrak{a}(E)=\mathfrak{a}(B)} g_{C, B}^{E}\{E\}  \qquad \text{ by Theorem~\ref{thm:cell structures} (P2, P3, P5, P7)}\\
	\equiv& \{C\}\cdot \{B\} \quad \mathrm{mod~} \Sc_{q}^{\geq \mathfrak{a}(B)+1}.
\end{align*}
\end{proof}

Let $R$ be an $\A$-algebra. Recall the convention in \S\ref{sec:sch} for the base change $R\otimes_\A-$.
\begin{prop}\label{b4}
	\begin{itemize}
		\item[(1)] Suppose $N$ is the length of the longest element in $W$. Then $(\mathrm{Ker} \Phi_{R})^{N+1}=0$.
		\item[(2)]  If $R=R_{0}[q, q^{-1}]$, where $R_{0}$ is an integral domain and $\A \rightarrow R$ is the obvious ring homomorphism, then $\mathrm{Ker}\Phi_{R}=0$. Particularly, the homomorphism $\Phi$ is injective.
	\end{itemize}
\end{prop}
\begin{proof}
$\mathrm{(1)}$ Note that for any $C\in \Xi$, we have $\mathfrak{a}(C)=\mathbf{a}(w_{C}^{+})\leq N$ by Proposition \ref{prop:d=d2}(4) and \cite[Proposition 13.8]{Lu03}. If $s\in \mathrm{Ker}\Phi_{R}$, then by Lemma \ref{b2}, we have $s\cdot \Sc_{q,R}^{\geq a}\subseteq \Sc_{q,R}^{\geq a+1}$ for any $a\in \mathbb{N}$. Applying this repeatedly, we see that $s_1s_2 \cdots s_{N+1}\in \Sc_{q,R}^{\geq N+1}=\{0\}$ for any $s_1,\ldots, s_{N+1}\in \mathrm{Ker}\Phi_{R}$.

$\mathrm{(2)}$ For any $b\in \Z$ and $p=\sum_{k\in \Z}a_{k}q^{k}\in R$, we define $\pi_{b}(p)=a_{b}\in R_{0}$. Take any $0\neq s=\sum_{A\in\Xi}p_{A}\{A\}\in \mathrm{Ker}\Phi_{R}$ ($p_{A}\in R$). Let $a=\min\{\mathfrak{a}(A)~|~A\in \Xi, p_{A}\neq 0\}\in \mathbb{N}$ and hence $\mathfrak{X} :=\{A\in \Xi\:|\:p_{A}\neq 0, \mathfrak{a}(A)=a\}$ is non-empty. Moreover, we can find $b\in \Z$ such that $p_{A}\in q^{b}R_{0}[q^{-1}]$ for all $A\in \mathfrak{X}$ and such that $\mathfrak{X}'=\{A\in \mathfrak{X}\:|\:\pi_{b}(p_{A})\neq 0\}$ is non-empty.

Let $C\in \mathfrak{X}'$. By Theorem \ref{thm:cell structures}, we can find $D=(\gamma, g, \gamma)\in \mathcal{D}$ such that $\gamma_{C, D}^{C^{t}}=\gamma_{C^{t}, C}^{D}\neq 0$ and $\mathfrak{a}(D)=\mathfrak{a}(C)=a$. We have $s\{D\}=\sum_{A\in\Xi}p_{A}\{A\}\{D\}$. Since $\{A\}\{D\}\in \Sc_{q,R}^{\geq \mathfrak{a}(A)}$, we have
\begin{align*}
s\{D\}=\sum_{A\in \mathfrak{X}}p_{A}\{A\}\{D\}\quad \mathrm{mod~} \Sc_{q,R}^{\geq a+1},
\end{align*} which implies $\sum_{A\in \mathfrak{X}}p_{A}\{A\}\{D\}\in \Sc_{q,R}^{\geq a+1}$ because of $s\{D\}\in \Sc_{q,R}^{\geq a+1}$ by Lemma \ref{b2}. In particular, the coefficient of $\{C\}$ is zero when expanding $\sum_{A\in \mathfrak{X}}p_{A}\{A\}\{D\}$. That is $\sum_{A\in \mathfrak{X}}p_{A}g_{A, D}^{C}=0$. Note that for $A\in \mathfrak{X}$, we have $p_{A}=\pi_{b}(p_{A})q^{b}+\mathrm{lower~terms}$, and $g_{A, D}^{C}=\gamma_{A, D}^{C^{t}}q^{a-\ell(w_\circ^\gamma)}+\mathrm{lower~terms}$. So the coefficient of $q^{a+b-\ell(w_\circ^\gamma)}$ in $\sum_{A\in \mathfrak{X}}p_{A}g_{A, D}^{C}$ is
\begin{align*}
\sum_{A\in \mathfrak{X}}\pi_{b}(p_{A})\gamma_{A, D}^{C^{t}}=\pi_{b}(p_{C})\gamma_{C, D}^{C^{t}}\neq 0,
\end{align*}
which is a contradiction. Thus $\mathrm{Ker}\Phi_R=0$.
Particularly, we have $\mathrm{Ker}\Phi=0$ by taking $R=\A$, and hence $\Phi$ is injective.
\end{proof}

Note that we regard $\Q$ as an $\A$-algebra via $q\mapsto1$.
\begin{cor}\label{b1}
Both $\Phi_\Q$ and $\Phi_\K$ are isomorphisms. In particular, $\mathbf{J}_\Q$ and $\mathbf{J}_\K$ are both split semisimple.
\end{cor}
\begin{proof}
	Take $R=\Q$. In this case, we obtain an algebra homomorphism $\Phi_{\Q} : \mathcal{S}_{\Q}\rightarrow \mathbf{J}_{\Q}$. Since $\mathcal{S}_{\Q}$ is split semisimple, by Proposition \ref{b4} we obtain $\mathrm{Ker}\Phi_{\Q}=0$, which implies the injectivity of $\Phi_{\Q}$. Note that $\mathrm{dim} \mathcal{S}_{\Q}=\mathrm{dim} \mathbf{J}_{\Q}=\sharp \Xi$. So $\Phi_{\Q}$ must be an isomorphism.
	
	Take $R=\K$. By Corollary \ref{b1} and the fact that injectivity is preserved by tensoring with a field of fractions, we know that $\Phi_{\K} : \Sc_{q,\K}\rightarrow \mathbf{J}_\K$ is injective. Since $\Sc_{q,\K}$ and $\mathbf{J}_\K$ have the same dimension, it follows that $\Phi_{\K}$ is an isomorphism.
\end{proof}

Denote by $\Phi_R^*$ the functor from $\mathrm{Mod}(\mathbf{J}_R)$ to $\mathrm{Mod}(\Sc_{q,R})$ induced by $\Phi_R$. Since $\Phi_\Q$ and $\Phi_\K$ are isomorphisms, $\Phi_\K^*$ and $\Phi_\Q^*$ are equivalent. As mentioned above, $\mathbf{J}_\Q$ is split semisimple, hence the functor $\K \otimes_\Q-$ gives an equivalence between $\mathrm{Mod}(\mathbf{J}_\Q)$ and $\mathrm{Mod}(\mathbf{J}_\K)$.


Let
$\{E^{\lambda}\:|\:\lambda\in \mathcal{E}\}$
be the complete set of pairwise non-isomorphic irreducible representations of $\mathcal{S}_{\Q}$, where $\mathcal{E}$ is a certain finite index set. For each $\lambda\in \mathcal{E}$, denote $\Phi_\Q^{*-1}(E^{\lambda})$ by $E_{\spadesuit}^{\lambda}$ and $\Phi_\K^*(E_{\spadesuit, \K}^{\lambda})$ by $E_{q}^{\lambda}$. Then $\{E_{\spadesuit}^{\lambda}\:|\:\lambda\in \mathcal{E}\}$ and $\{E_{q}^{\lambda}\:|\:\lambda\in \mathcal{E}\}$ forms the complete sets of pairwise non-isomorphic irreducible representations of $\mathbf{J}_{\Q}$ and $\Sc_{q,\K}$ respectively.
For $\mathbf{c}\in B_0$, let $\mathcal{E}_{\mathbf{c}}\subset\mathcal{E}$ be the set consisting of those $\lambda$ such that $t_{C}E_{\spadesuit}^\lambda \neq 0$ for some $C\in \mathbf{c}$.
\begin{lem}\label{lem:part}
 The set $\mathcal{E}$ can be partitioned as
$\mathcal{E}=\coprod_{\mathbf{c}\in B_0}\mathcal{E}_{\mathbf{c}}$.
\end{lem}
\begin{proof}
Since $\mathbf{J}=\bigoplus_{\mathbf{c}\in B_{0}}\mathbf{J}_{\mathbf{c}}$, the conclusion is clear.
\end{proof}

Since $\Sc_{q,\K}$ is split semisimple, we have $$\Sc_{q,\K}=\bigoplus_{\lambda\in\mathcal{E}}\Sc_{q,\K}^\lambda,\quad\mbox{where~}\Sc_{q,\K}^\lambda\cong\mathrm{End}_\K(E_q^\lambda).$$

\subsection{Schur elements}
Recall that in \S\ref{sec:abiform} we constructed a non-degenerate symmetric bilinear form $(\cdot, \cdot)$ on $\Sc_{q,\K}$. Because of \cite[Proposition~19.2]{Lu03}, we have the following orthogonality relation:
\begin{align}\label{d2}
	\sum\limits_{A=(\gamma,g,\nu) \in \Xi} q^{\ell(w_\circ^\gamma)+\ell(w_\circ^\nu)-2\ell(w_A^+)}\mathfrak{P}_{W_\gamma \cap g W_\nu g^{-1}}\Tr([A], E_q^{\lambda}) \Tr([A^{t}], E_q^{\lambda'})=\delta_{\lambda\lambda'} d_{\lambda}p_{\lambda},
\end{align}
where $d_{\lambda}=\mathrm{dim} E^{\lambda}$ and $p_\lambda \in \A$. By the property of associative symmetric bilinear form over split semisimple algebras, we have $p_\lambda=(e,e)^{-1}$ for an arbitrary primitive idempotent $e\in\Sc_{q,\K}^\lambda$. These elements $p_\lambda\ (\lambda \in \mathcal{E})$ are called \emph{Schur elements} of $\Sc_{q,\K}$ with respect to $(\cdot, \cdot)$.

Similarly, we can define Schur elements of $\HH_\K$ (resp. $\Sc'_{q,\K}$) with respect to $(\cdot,\cdot)_{1}$ (resp. $(\cdot,\cdot)_3$) introduced in \S\ref{sec:abiform}.

\begin{prop}\label{scel}
Each Schur element of $\Sc_{q,\K}$ must be a Schur element of $\HH_\K$.
Furthermore, the set of Schur elements of $\Sc_{q,\K}$ coincides with that of $\HH_\K$ if $\Lambda_\mathbf{f}$ contains at least one regular $W$-orbit.
\end{prop}
\begin{proof}
We have $\HH_\K = \bigoplus_\lambda \HH_\K^\lambda \cong \bigoplus_\lambda \End_\K (V^\lambda)$, where $\{V^\lambda\}_\lambda$ is a complete set of simple modules of $\HH_\K$. Then $$\Sc_{q,\K}' \cong (\bigoplus_\lambda \HH_\K^\lambda) \otimes_\K \mathcal{M}_{\ff,\K} \cong \bigoplus_\lambda \big(\HH_\K^\lambda \otimes_\K \mathcal{M}_{\ff,\K}\big).$$

For any primitive idempotent element $e\in \HH_\K^\lambda$, we can obtain a primitive idempotent element $e \otimes e_{\gamma\gamma}\in \HH_\K^\lambda \otimes_\K \mathcal{M}_{\ff,\K}$ for an arbitrary $\gamma \in \Lambda_{\ff}$. Moreover, $(e \otimes e_{\gamma\gamma},e \otimes e_{\gamma\gamma})_3=(e,e)_1$. Therefore, each Schur element of $\Sc'_{q,\K}$ must be one of $\HH_\K$.

Recall from \S\ref{sec:abiform} that $\Sc_{q,\K} \cong m_\ff \Sc_{q,\K}' m_\ff$, which is employed to induce the bilinear form $(\cdot,\cdot)$ from $(\cdot,\cdot)_3$. Note that a primitive idempotent element of $m_\ff \Sc_{q,\K}' m_\ff$ is also one of $\Sc_{q,\K}'$. Thus a Schur element of $\Sc_{q,\K}$ is also one of $\Sc'_{q,\K}$, and hence one of $\HH_\K$.

Assume $\Lambda_{\ff}$ contains a regular orbit $\gamma$. Then $\HH_\K \cong m_\gamma \Sc_{q,\K}' m_\gamma \subset m_\ff \Sc_{q,\K}' m_\ff$. Thus, by the similar argument every Schur element of $\HH_\K$ is one of $m_\gamma \Sc_{q,\K}' m_\gamma= \Sc_{q,\K}$.
\end{proof}

\subsection{The integer $f_\lambda$ and good primes}
Denote by $f_\lambda$ the coefficient of the lowest $q$-power term of $p_\lambda$.
We say that a prime number $p$ is \emph{bad} if $p$ divides $f_{\lambda}$ for some $\lambda \in \mathcal{E}$; otherwise, $p$ is \emph{good}.

Thanks to Proposition \ref{scel}, if $\Lambda_\mathbf{f}$ contains at least one regular $W$-orbit, the sufficient and necessary condition for $p$
being good is as the same as that for Hecke algebras listed below (cf. \cite[\S2.2]{Ge07}):
$$
A_n: \mbox{no condition}; \quad
B_n,C_n,D_n:  p\neq2;\quad
G_2,F_4,E_6,E_7:   p\neq2,3; \quad
E_8: p\neq2,3,5;
$$
if $\Lambda_\mathbf{f}$ contains no regular $W$-orbit, the above is still a sufficient condition for $p$ being good.

\subsection{The integer $\mathbf{a}_\lambda$}
\label{sec:ainvaritna}
\begin{lem}\label{integ}
	\begin{itemize}
		\item[(1)] Let $M$ be a $\mathbf{J}_\Q$-module of dimension $m$. Then there exists a basis of $M$ such that $\mathbf{J}$ maps into $\Mat_m(\Z)$ under this basis.
		\item[(2)] Let $M$ be a $\Sc_{q,\K}$-module of dimension $m$. Then there exists a basis of $M$ such that $\Sc_{q}$ maps into $\Mat_m(\A)$ under this basis.
	\end{itemize}
\end{lem}
\begin{proof}
	Item (1) follows from a general result described in \cite[\S7.3.7]{GP00}.

By the discussion at the end of \S\ref{sec:phipro}, we know that $M \cong \Phi_{\K}^*(M_\K')$ for some $\mathbf{J}_\Q$-module $M'$ of dimension $m$. Noting that $\Phi$ maps $\Sc$ into $\mathbf{J}_{\A}$, it is clear that (2) follows from (1).
\end{proof}

Thus for $\lambda \in \mathcal{E}$, we can denote by $\mathbf{a}_{\lambda}\in \Z$ the unique integer such that
$\Tr([A], E_{q}^{\lambda})\in q^{-\mathbf{a}_{\lambda}+\ell(w_\circ^{\co(A)})}\Z[q]$ for all $A \in \Xi$, but $\Tr([A], E_{q}^{\lambda})\notin q^{-\mathbf{a}_{\lambda}+\ell(w_\circ^{\co(A)})+1}\Z[q]$ for some $A \in \Xi$. For $\mathbf{c}\in B_0$,  let $\mathfrak{a}(\mathbf{c})$ be the common value of $\mathfrak{a}(A)$ for $A \in \mathbf{c}$ (cf. Theorem~\ref{thm:cell structures}(P4)).
\begin{prop}\label{c1}
\begin{itemize}
	\item[(1)] If $\lambda \in \mathcal{E}_{\mathbf{c}}$, then $\mathbf{a}_{\lambda}=\mathfrak{a}(\mathbf{c})$.
	\item[(2)] The constant term of $q^{\mathbf{a}_{\lambda}-\ell(w_\circ^{\co(A)})}\Tr([A], E_{q}^{\lambda})$ is $\Tr(t_A, E_{\spadesuit}^{\lambda})$ for all $A \in \Xi$.
\end{itemize}

\end{prop}
\begin{proof}
For any $A\in\Xi$,
\begin{align*}
	\Tr(\{A\}, E_{q}^{\lambda})= \Tr(\Phi_{\K}(\{A\}), E_{\spadesuit,\K}^{\lambda})=\sum\limits_{\mathbf{c'}\in B_{0}, D\in \mathcal{D}_{\mathbf{c'}}, B\in \mathbf{c'}}g_{A, D}^{B}\Tr(t_{B}, E_{\spadesuit}^{\lambda}).
\end{align*}
Since $E_{\spadesuit}^{\lambda}\in \mathcal{E}_{\mathbf{c}}$, we have $\Tr(t_{B}, E_{\spadesuit}^{\lambda})=0$ unless $B\in \mathbf{c}$. Thus, we have
\begin{align*}
	\Tr(\{A\}, E_{q}^{\lambda})=&\sum\limits_{D\in \mathcal{D}_{\mathbf{c}}, B\in \mathbf{c}}g_{A, D}^{B}\Tr(t_{B}, E_{\spadesuit}^{\lambda})\\
	\equiv & \bigg(\sum\limits_{D\in \mathcal{D}_{\mathbf{c}}, B\in \mathbf{c}} \gamma_{A, D}^{B^{t}}\Tr(t_{B}, E_{\spadesuit}^{\lambda})\bigg)q^{-\mathfrak{a}(\mathbf{c})+\ell(w_\circ^{\co(A)})} \quad \mathrm{mod~} q^{-\mathfrak{a}(\mathbf{c})+\ell(w_\circ^{\co(A)})+1}\Z[q]\\
	=&\Tr(t_{A}, E_{\spadesuit}^{\lambda})q^{-\mathfrak{a}(\mathbf{c})+\ell(w_\circ^{\co(A)})} \quad \mathrm{mod~} q^{-\mathfrak{a}(\mathbf{c})+\ell(w_\circ^{\co(A)})+1}\Z[q].
\end{align*}
Here the last equality follows from Theorem~\ref{thm:cell structures}(P2, P3, P5, P7).	

It follows from \eqref{eq:can} that $[A]=\{A\}+\sum_{A'<A} \Z[q]\{ A'\}$.
	Therefore,
	\begin{align*}
		\Tr([A], E_{q}^{\lambda})\equiv \Tr(\{A\}, E_{q}^{\lambda})\equiv\Tr(t_{A}, E_{\spadesuit}^{\lambda})q^{-\mathfrak{a}(\mathbf{c})+\ell(w_\circ^{\co(A)})} \quad \mathrm{mod~} q^{-\mathfrak{a}(\mathbf{c})+\ell(w_\circ^{\co(A)})+1}\Z[q].
	\end{align*}
	Since $\lambda \in \mathcal{E}_{\mathbf{c}}$, we have $\Tr(t_A, E_{\spadesuit}^{\lambda})\neq 0$ for some $A\in \mathbf{c}$. Thus, the above formula shows that $\mathbf{a}_{\lambda}=\mathfrak{a}(\mathbf{c})$, and hence the statement (2).
\end{proof}

\begin{cor}\label{c6}
	It holds that
	\begin{align*}
		p_\lambda=f_\lambda q^{-2\mathbf{a}_\lambda}+\mbox{higher $q$-power terms},\quad\mbox{and}\quad
d_\lambda f_\lambda=\sum_{A \in \Xi} \Tr(t_A,E_{\spadesuit}^\lambda)\Tr(t_{A^t},E_{\spadesuit}^\lambda).
	\end{align*}
\end{cor}

\begin{proof}
	The statement follows from a direct computation as follows:
	\begin{align*}
		d_\lambda p_\lambda=&\sum_{A=(\gamma,g,\nu)  \in \Xi} q^{\ell(w_\circ^\gamma)+\ell(w_\circ^\nu)-2\ell(w_A^+)}\mathfrak{P}_{W_\gamma \cap g W_\nu g^{-1}}\Tr([A],E_q^\lambda)\Tr([A^t],E_q^\lambda)\\
		\equiv&\sum_{A \in \Xi} q^{-\ell(w_\circ^{\ro(A)})-\ell(w_\circ^{\co(A)})}\Tr([A],E_q^\lambda)\Tr([A^t],E_q^\lambda) \mod q^{-2\mathbf{a}_\lambda+1}\mathbb{Z}[q]\\
		=&\sum_{A \in \Xi} \Tr(t_A,E_{\spadesuit}^\lambda)\Tr(t_{A^t},E_{\spadesuit}^\lambda)q^{-2\mathbf{a}_\lambda} \mod q^{-2\mathbf{a}_\lambda+1}\mathbb{Z}[q]\quad \mbox{(By Proposition \ref{c1})}\\
		=&d_\lambda f_\lambda q^{-2\mathbf{a}_\lambda} \mod q^{-2\mathbf{a}_\lambda+1}\mathbb{Z}[q].
	\end{align*}
\end{proof}

\section{Cellularity of $q$-Schur algebras}
\label{sec:cellularity}
In \cite{Ge07}, Geck showed that Hecke algebras of finite type are cellular in the sense of Graham and Lehrer \cite{GL96}. In this section, we shall follow Geck's approach to prove that $q$-Schur algebras of finite type are cellular, too.

\subsection{The matrix representation $\rho^\lambda$ and the matrix $P^\lambda$}
\label{subsec:rep}
Recall that in \S\ref{sec:phipro}, we define a non-degenerate symmetric associative bilinear form $(\cdot,\cdot)$ on $\mathbf{J}$, with respect to which $\{t_A ~|~ A \in \Xi\}$ and $\{t_{A^t} ~|~ A \in \Xi\}$ form a pair of dual bases. Choose a $\Q$-basis of $E_{\spadesuit}^{\lambda}$. We obtain a matrix representation
\begin{align*}
	\rho^\lambda:\mathbf{J}_{\Q} \rightarrow \Mat_{d_\lambda}(\mathbb{\Q}), \quad t_A \mapsto (\rho_\mathfrak{st}^\lambda(t_A))_{\mathfrak{s,t} \in M(\lambda)}, \quad\mbox{where $M(\lambda)=\{1,\ldots, d_\lambda\}$}.
\end{align*}
Moreover, thanks to Lemma~\ref{integ}, we can let $\rho^\lambda(\mathbf{J}) \subset \Mat_{d_\lambda}(\mathbb{\Z})$.
For $\lambda, \lambda' \in \mathcal{E}, \mathfrak{s,t} \in M(\lambda), \mathfrak{u,v} \in M(\lambda')$, by Corollary \ref{c6} and \cite[Corollary 7.2.2]{GP00}, the following Schur relation holds:
\begin{align}
	\label{c2}
	\sum_{A\in \Xi} \rho_\mathfrak{st}^\lambda(t_A) \rho_\mathfrak{uv}^{\lambda'}(t_{A^t})=\delta_{(\lambda,\mathfrak{s,t}),(\lambda',\mathfrak{v,u})} f_\lambda,
\end{align}
which is equivalent to
\begin{align}\label{c3}
	\sum_{\lambda\in \mathcal{E}} \sum_{\mathfrak{s} \in M(\lambda)} \frac{1}{f_\lambda}(\rho^\lambda(t_A) \rho^{\lambda}(t_{B}))_{\mathfrak{s,s}}=\sum_{\lambda\in \mathcal{E}} \sum_{\mathfrak{s,t} \in M(\lambda)} \frac{1}{f_\lambda}\rho_\mathfrak{st}^\lambda(t_A) \rho_\mathfrak{ts}^{\lambda}(t_{B})
	=\delta_{A, B^t}.
\end{align}

We have the following proposition, whose proof is similar to the one of \cite[Proposition 2.6]{Ge07}.
%


\begin{prop}\label{c4}
	There exists a symmetric, positive-definite matrix
	$$P^\lambda=(\beta_\mathfrak{st}^\lambda)_{\mathfrak{s,t} \in M(\lambda)}\in \mathrm{Mat}_{d_\lambda}(\Z)$$
	such that the following two conditions hold:
	\begin{itemize}
		\item[(a)] $P^\lambda \cdot \rho^\lambda(t_{A^t})=(\rho^\lambda(t_A))^{\mathrm{tr}} \cdot P^\lambda$ for all $A \in \Xi$, where $\mathrm{tr}$ means the transpose;
		\item[(b)] Any prime which divides $\det(P^\lambda)$ is bad.
	\end{itemize}
\end{prop}
\begin{proof} For convenience, we write $E=E_{\spadesuit}^{\lambda}$, $d=d_\lambda$ and $\rho=\rho^\lambda$.
	
	We set $$P_1=\sum_{B \in \Xi} \rho(t_B)^{\mathrm{tr}}  \cdot\rho(t_B) \in \Mat_d(\Z),$$ which is clearly a symmetric matrix. Take any nonzero vector $\alpha=(a_1,\ldots,a_d) \in \Z^d$. Since $\rho$ is irreducible, there exists some $B \in \Xi$ such that $\alpha\cdot\rho(t_B)^{\mathrm{tr}} \neq 0$. Hence $eP_1 e^{\mathrm{tr}}>0$, which means that $P_1$ is positive-definite. For any $A \in \Xi$, by Theorem~\ref{thm:cell structures}(P7) and Corollary \ref{lem:d=d3}(2) we have
	\begin{align*}
		P_1 \cdot \rho(t_{A^t})&=\sum_B \rho(t_B)^{\mathrm{tr}} \cdot \rho(t_B t_{A^t})=\sum_{B,C} \gamma_{B,A^t}^{C^t} \rho(t_B)^{\mathrm{tr}} \cdot \rho(t_C)\\
		&=\sum_{B,C} \gamma_{A^t,C^t}^B \rho(t_B)^{\mathrm{tr}} \cdot \rho(t_C)=\sum_{B,C} \gamma_{C,A}^{B^t} \rho(t_B)^{\mathrm{tr}} \cdot \rho(t_C)\\
		&=\sum_C \rho(t_C t_A)^{\mathrm{tr}} \cdot \rho(t_C)=\rho(t_A)^{\mathrm{tr}} \cdot P_1.
	\end{align*}
	
	Let $n \in \N$ be the greatest common divisor of all entries of $P_1$ and let $P=n^{-1}P_1 \in \Mat_{d}(\Z)$. Then $P$ is a symmetric, positive-definite matrix such that the condition (a) holds. It remains to prove that $P$ satisfies the condition $\mathrm{(b)}$. Let $p$ be a prime number and denote by $\bar{P}$ the matrix obtained from $P$ by reducing all entries modulo $p$. By reduction modulo $p$, we also obtain an $\F_p$-algebra $\mathbf{J}_p=\F_p \otimes_\Z \mathbf{J}$ and a corresponding matrix representation $\bar{\rho}: \mathbf{J}_p \rightarrow \Mat_{d}(\F_p)$. Hence we have $$\bar{P} \neq 0 \quad \mbox{and} \quad \bar{P} \cdot \bar{\rho}(t_{A})=\bar{\rho}(t_{A^t})^{\mathrm{tr}} \cdot \bar{P} \quad \mbox{for all} ~ A \in \Xi.$$
	
	By Corollary \ref{lem:d=d3}(2) again, we see that the map $t_A \mapsto t_{A^t}$ defines an anti-involution on $\mathbf{J}$. Hence the assignment $t_A \mapsto \bar{\rho}(t_{A^t})^{\mathrm{tr}}$ also defines a representation of $\mathbf{J}_p$, and the above identity shows that $\bar{P} \neq 0$ is an ``intertwining operator". Hence, if we knew that $\bar{\rho}$ was irreducible, then Schur's Lemma would imply that $\bar{P}$ was invertible and so $p$ could not divide $\det(P)$.
	
	Thus it remains to show that $\bar{\rho}$ is an irreducible representation of $\mathbf{J}_p$ whenever $p$ is good. But this follows from a general argument about symmetric algebras, noting that $\mathbf{J}$ is symmetric (cf. \cite[Remark 7.2.3]{GP00}).
\end{proof}

\subsection{Cellular algebras}
We now fix a ring $R \subset \C$ where all bad primes are invertible. Denote $\underline{\A}=R[q, q^{-1}]$. In the sequel of this section, we shall consider the $q$-Schur algebras $\Sc_{q,\underline{\A}}$.

In order to show that $\Sc_{q,\underline{\A}}$ is cellular in the sense of Graham-Lehrer \cite[Definition 1.1]{GL96}, we must
specify a quadruple $(\mathcal{E}, M, C, \Psi)$ satisfying the following conditions.

$\mathrm{(C1)}$ $\mathcal{E}$ is a partially ordered set with a partial order $\prec$, $\{M(\lambda)\:|\:\lambda\in \mathcal{E}\}$ is a collection of finite
sets and
\begin{align*}
C: \coprod_{\lambda\in \mathcal{E}} M(\lambda) \times M(\lambda)\rightarrow \Sc_{q,\underline{\A}}
\end{align*}
is an injective map whose image is an $\underline{\A}$-basis of $\Sc_{q,\underline{\A}}$;

$\mathrm{(C2)}$ If $\lambda\in \mathcal{E}$ and $\mathfrak{s,t} \in M(\lambda)$, write $C(\mathfrak{s}, \mathfrak{t})=C_{\mathfrak{s}, \mathfrak{t}}^{\lambda}\in \Sc_{q}$. Then $\Psi: \Sc_{q,\underline{\A}} \rightarrow \Sc_{q,\underline{\A}}$ is an $\underline{\A}$-linear anti-involution such that $\Psi(C_{\mathfrak{s}, \mathfrak{t}}^{\lambda})=C_{\mathfrak{t}, \mathfrak{s}}^{\lambda}$.

$\mathrm{(C3)}$ If $\lambda\in \mathcal{E}$ and $\mathfrak{s,t} \in M(\lambda)$, then for any element $h \in \Sc_{q,\underline{\A}}$,
\begin{align*}
hC_{\mathfrak{s}, \mathfrak{t}}^{\lambda}\equiv \sum\limits_{\mathfrak{s}'\in M(\lambda)} r_{h}(\mathfrak{s}', \mathfrak{s})C_{\mathfrak{s}', \mathfrak{t}}^{\lambda}\quad \mathrm{mod}~\Sc_{q}(\prec\lambda),
\end{align*}
where $r_{h}(\mathfrak{s}', \mathfrak{s})\in \underline{\A}$ is independent of $\mathfrak{t}$, and $\Sc_{q}(\prec\lambda)$ is the
$\underline{\A}$-submodule of $\Sc_{q}$ generated by $\{C_{\mathfrak{s}'', \mathfrak{t}''}^{\lambda'}\:|\:\lambda'\prec\lambda; \mathfrak{s}'', \mathfrak{t}''\in M(\lambda')\}$.

\subsection{Cellularity}
We now define a required quadruple $(\mathcal{E}, M, C, \Psi)$ as follows.

Let $\mathcal{E}$ be the indexing set introduced in \S\ref{sec:phipro}. Using the $\mathbf{a}$-invariants in Proposition \ref{c1}, we define a partial order $\prec$ on $\mathcal{E}$ by $\lambda \prec \lambda'$ if and only if $\mathbf{a}_{\lambda}> \mathbf{a}_{\lambda'}$. An $\underline{\A}$-linear anti-involution $\Psi$ on $\Sc_{q,\underline{\A}}$ can be induced from the one introduced in \eqref{eq:psi}. For $\lambda\in \mathcal{E}$, we set $M(\lambda)=\{1,\ldots, d_\lambda\}$ as before.

Let $\{\rho_\mathfrak{st}^\lambda(t_A)\}$ and $\{\beta_\mathfrak{st}^\lambda\}$ be defined as in \S\ref{subsec:rep}. For $\lambda\in \mathcal{E}$ and $\mathfrak{s,t} \in M(\lambda)$, we define
\begin{equation}\label{cst}
C_\mathfrak{s,t}^\lambda:=\sum_{A\in \Xi} (P^\lambda \cdot \rho^\lambda(t_{A^t}))_\mathfrak{t,s} \{A\}=\sum_{A \in \Xi} \sum_{\mathfrak{u}\in M(\lambda)} \beta_\mathfrak{tu}^\lambda \rho_\mathfrak{us}^\lambda(t_{A^t})\{A\}.
\end{equation}

\begin{thm}\label{e5}
The Schur algebra $\Sc_{q,\underline{\A}}$ is cellular and the quadruple $(\mathcal{E}, M, C, \Psi)$ is the desired cell datum.
\end{thm}
\begin{proof}
Assume $E_{\spadesuit}^{\lambda}\in \mathcal{E}_{\mathbf{c}}$ for some fixed $\mathbf{c}\in B_0$. We have $\rho^\lambda(t_{A^{t}})=0$ unless $A^{t}\in \mathbf{c}$ (i.e. $A\in \mathbf{c}$). In case $A\in \mathbf{c}$, we always have $\mathfrak{a}(A)=\mathbf{a}_{\lambda}$ by Proposition \ref{c1}. Thus Proposition \ref{c4} implies that $C_\mathfrak{s,t}^\lambda$ is a $\Z$-linear combination of the elements $\{A\}$ with $\mathfrak{a}(A)=\mathbf{a}_{\lambda}$.

\noindent	
{\em \underline{Step 1}}. We first prove that the elements $\{C_\mathfrak{s,t}^\lambda\}$ span $\Sc_{q,\underline{\A}}$ as an $\underline{\A}$-module. For any $B \in \Xi$, it is a straightforward calculation that
\begin{align*}
\{B\}&=\sum_{A\in\Xi} \delta_{A, B} \{A\}=\sum_{A\in \Xi} \sum_{\lambda \in \mathcal{E}} \sum_{\mathfrak{s} \in M(\lambda)} \frac{1}{f_\lambda}(\rho^\lambda(t_B) \cdot \rho^\lambda(t_{A^t}))_\mathfrak{s,s} \{A\}\qquad \mbox{by \eqref{c3}}\\
&=\sum_{\lambda \in \mathcal{E}} \sum_{\mathfrak{s,t} \in M(\lambda)} \frac{1}{f_\lambda}(\rho^\lambda(t_B) \cdot (P^\lambda)^{-1})_\mathfrak{s,t} C_\mathfrak{s,t}^\lambda \qquad\mbox{by \eqref{cst}}.
\end{align*}
Note that the coefficients $\frac{1}{f_\lambda}(\rho^\lambda(t_B) \cdot (P^\lambda)^{-1})_\mathfrak{s,t}\in R$ since $f_\lambda$ and $\det(P^\lambda)$ are invertible in $R$.
Thus $\{B\}$ is an $R$-linear combination of the elements $C_\mathfrak{s,t}^\lambda$ as required; moreover, if the coefficient of some $C_\mathfrak{s,t}^\lambda$ is nonzero, we must have $\mathbf{a}_{\lambda}=\mathfrak{a}(B)$.

Furthermore, by Lemma~\ref{b1} and Wedderburn's Theorem, we have $\mathrm{rank}_{\underline{\A}}(\Sc_{q,\underline{\A}})=\dim_{\K} \Sc_{q,\K}=\sum_\lambda d_\lambda^2$. Therefore, the elements $\{C_\mathfrak{s,t}^\lambda\:|\: \lambda \in \mathcal{E} ~\mbox{and}~ \mathfrak{s,t} \in M(\lambda)\}$ are linearly independent over $\underline{\A}$ and hence form an $\underline{\A}$-basis of $\Sc_{q,\underline{\A}}$. Thus $\mathrm{(C1)}$ holds.
	
\vspace{0.3cm}
\noindent
{\em \underline{Step 2}}. Note that $\Psi(\{A\})=\{A^t\}$ by \cite[Lemma~3.6]{CLW20}.
We check directly that
\begin{align*}
\Psi(C_\mathfrak{s,t}^\lambda)&=\sum_A (P^\lambda \cdot \rho^\lambda(t_{A^t}))_\mathfrak{t,s} \{A^t\}=\sum_A (\rho^\lambda(t_{A})^{\mathrm{tr}} \cdot P^\lambda)_\mathfrak{t,s} \{A^t\}\\&=\sum_A (P^\lambda \cdot \rho^\lambda(t_{A}))_\mathfrak{s,t} \{A^t\}=C_\mathfrak{t,s}^\lambda.
\end{align*}
Therefore, $\mathrm{(C2)}$ holds.
	
\vspace{0.3cm}
\noindent
{\em \underline{Step 3}}. As mentioned in \S\ref{sec:asymp}, we equip $\Sc_{q,\underline{\A}}$ with a $\mathbf{J}_{\underline{\A}}$-module structure via $$t_{A}\ast \{B\}=\sum\limits_{C\in \Xi}\gamma_{A,B}^{C^{t}}\{C\}.$$
For $A\in \Xi$, we claim that
\begin{align*}
t_A\ast C_\mathfrak{s,t}^\lambda=\sum_{\mathfrak{s'} \in M(\lambda)} \rho_\mathfrak{s's}^\lambda(t_A) C_\mathfrak{s',t}^\lambda.
\end{align*}
In fact,
\begin{align*}
t_A\ast C_\mathfrak{s,t}^\lambda=&\sum_{B \in \Xi} \sum_{\mathfrak{u}\in M(\lambda)} \beta_\mathfrak{tu}^\lambda \rho_\mathfrak{us}^\lambda(t_{B^t})t_A\ast \{B\}
=\sum_{B, C \in \Xi} \sum_{\mathfrak{u}\in M(\lambda)} \beta_\mathfrak{tu}^\lambda \rho_\mathfrak{us}^\lambda(t_{B^t})\gamma_{A,B}^{C^{t}}\{C\}\\
=&\sum_{C \in \Xi} \sum_{\mathfrak{u}\in M(\lambda)} \beta_\mathfrak{tu}^\lambda \rho_\mathfrak{us}^\lambda\Big(\sum_{B \in \Xi} \gamma_{C^{t}, A}^{B} t_{B^t}\Big)\{C\}\qquad\qquad \text{ by Theorem~\ref{thm:cell structures}(P7)}\\
=&\sum_{C \in \Xi} \sum_{\mathfrak{u}\in M(\lambda)} \beta_\mathfrak{tu}^\lambda \rho_\mathfrak{us}^\lambda(t_{C^{t}}t_A)\{C\}
=\sum_{C \in \Xi} (P^\lambda \cdot \rho^\lambda(t_{C^{t}}t_A))_\mathfrak{t,s}\{C\}\\
=&\sum_{C \in \Xi} \sum_{\mathfrak{s'} \in M(\lambda)} (P^\lambda \cdot \rho^\lambda(t_{C^{t}}))_{\mathfrak{t}, \mathfrak{s}'}  \rho_{\mathfrak{s}'\mathfrak{s}}^\lambda(t_A)\{C\}
=\sum_{\mathfrak{s'} \in M(\lambda)} \rho_\mathfrak{s's}^\lambda(t_A) C_\mathfrak{s',t}^\lambda,
\end{align*}
as claimed.

For any element $h\in \Sc_{q}$, we write $\Phi_{\underline{\A}}(h)=\sum_{A\in \Xi}a_{h}(A)t_{A}$, where $a_{h}(A)\in \underline{\A}$. Define
\begin{align*}
r_{h}(\mathfrak{s'}, \mathfrak{s})=\sum\limits_{A\in \Xi}a_{h}(A)\rho_\mathfrak{s's}^\lambda(t_A)\quad \text{for any $\mathfrak{s}', \mathfrak{s}\in M(\lambda)$}.
\end{align*}
Note that $r_{h}(\mathfrak{s'}, \mathfrak{s})\in \underline{\A}$ only depends on $\mathfrak{s}', \mathfrak{s}$ and $h$. Moreover, the above calculation shows that
\begin{align*}
\Phi_{\underline{\A}}(h)\ast C_\mathfrak{s,t}^\lambda=\sum_{\mathfrak{s'} \in M(\lambda)}r_{h}(\mathfrak{s'}, \mathfrak{s})C_\mathfrak{s',t}^\lambda.
\end{align*}

At the beginning of the proof we said that $C_\mathfrak{s,t}^\lambda$ is a $\Z$-linear combination of the elements $\{A\}$ with $\mathfrak{a}(A)=\mathbf{a}_{\lambda}$. Thus, by Lemma~\ref{b2} we immediately obtain
\begin{align*}
hC_{\mathfrak{s}, \mathfrak{t}}^{\lambda}\equiv \Phi_{\underline{\A}}(h)\ast C_\mathfrak{s,t}^\lambda=\sum\limits_{\mathfrak{s}'\in M(\lambda)} r_{h}(\mathfrak{s}', \mathfrak{s})C_{\mathfrak{s}', \mathfrak{t}}^{\lambda}\quad \mathrm{mod~}\Sc_{q,\underline{\A}}^{\geq \mathbf{a}_{\lambda}+1}.
\end{align*}
From {\em Step 1}, we see that each element $\{B\}$ is an $R$-linear combination of the elements $C_\mathfrak{s,t}^\lambda$ with $\mathbf{a}_{\lambda}=\mathfrak{a}(B)$. Therefore, each element in $\Sc_{q}^{\geq \mathbf{a}_{\lambda}+1}$ is an $\underline{\A}$-linear combination of the elements $C_\mathfrak{s',t'}^{\lambda'}$ with $\mathbf{a}_{\lambda'}\geq \mathbf{a}_{\lambda}+1$. Hence $\mathrm{(C3)}$ holds.
\end{proof}

\begin{cor}\label{f5}
Let $k$ be an integral domain with a ring homomorphism $\underline{\A} \rightarrow k$. Then the above ingredients define a cell structure on $\Sc_{q, k}$.
\end{cor}

\subsection{Application: Specht modules}
Analogous to the discussions in \cite[Example 4.4]{Ge07}, we can follow a general approach proposed in \cite{GL96} to develop a theory of Specht modules for $\Sc_{q, k}$ over a field $k$, assuming that there exists a ring homomorphism from $\underline{\A}$ to $k$. More concretely, we can use the cell datum to define a cell representation $W_{k}(\lambda)$ of $\Sc_{q, k}$ for each $\lambda\in \mathcal{E}$, and moreover, on each $W_{k}(\lambda)$ we can define a symmetric bilinear form $g^{\lambda}_{k}$.

Denote $\mathcal{E}^{\circ} :=\{\lambda\in \mathcal{E}\:|\:g^{\lambda}_{k}\neq 0\}$ and let $L_{k}(\lambda) :=W_{k}(\lambda)/\mathrm{rad}(g^{\lambda}_{k})$ for each $\lambda\in \mathcal{E}^{\circ}$. Thanks to \cite[Theorem 3.4]{GL96}, the set $\{L_{k}(\lambda)\:|\:\lambda\in \mathcal{E}^{\circ}\}$ is a complete set of absolutely irreducible representations of $\Sc_{q, k}$ (up to isomorphism). Moreover, by \cite[Theorem 3.8 \& Remark 3.10]{GL96}, we see that
\begin{itemize}
  \item[(1)] $\Sc_{q, k}$ is quasi-hereditary if $\mathcal{E}=\mathcal{E}^{\circ}$;
  \item[(2)] $\Sc_{q, k}$ is semisimple if and only if $\mathcal{E}=\mathcal{E}^{\circ}$ and the bilinear form $g^{\lambda}_{k}$ is nondegenerate for each $\lambda\in \mathcal{E}$.
\end{itemize}

Denote by $[W_{k}(\lambda) : L_{k}(\lambda')]$ the multiplicity of $L_{k}(\lambda')$ as a composition factor of $W_{k}(\lambda)$. Then, by \cite[Proposition 3.6]{GL96} we have
\begin{align*}
[W_{k}(\lambda) : L_{k}(\lambda)]=1&\quad\mbox{ for any }\lambda \in \mathcal{E}^{\circ},\quad\mbox{and}\\
[W_{k}(\lambda) : L_{k}(\lambda')]=0&\quad\mbox{ unless }\lambda \preceq \lambda' \text{ (i.e., $\lambda=\lambda'$ or $\mathbf{a}_{\lambda}> \mathbf{a}_{\lambda'}$).}
\end{align*}
So the decomposition matrix $D=([W_{k}(\lambda) : L_{k}(\lambda')])_{\lambda\in \mathcal{E}, \lambda' \in \mathcal{E}^{\circ}}$ is lower uni-triangular if the rows and columns are ordered according to increasing $\mathbf{a}$-values.

\section{Representations and beyond}
\label{sec:blocks}

In this section we shall define and study special modules for the asymptotic Schur algebras. Moreover, we shall study the blocks of $\Sc_q(X_\ff)$ and left cell representations of the Schur algebras.

\subsection{The structure of $\mathbf{J}_{\mathbf{c}}$}
Fix a two-sided cell $\mathbf{c}\in B_0$. Let
$\mathfrak{c}$
be the corresponding two-sided cell containing $\{w_C^+ ~|~ C \in \mathbf{c}\}$ in the Weyl group $W$ by Proposition \ref{prop:d=d2}(3).
Denote
\[
\mathfrak{L}:= \{\text{left cells contained in } \mathfrak{c}\},\quad \mathbb{L}:= \{\text{left cells contained in } \mathbf{c}\}..
\]
We have \[\mathfrak{c}=\bigsqcup_{\mathfrak{\Gamma} \in \mathfrak{L}}\mathfrak{\Gamma}=\bigsqcup_{\mathfrak{\Gamma}, \mathfrak{\Gamma'}\in \mathfrak{L}} (\mathfrak{\Gamma}\cap \mathfrak{\Gamma'}^{-1}), \quad \mathbf{c}=\bigsqcup_{\Gamma \in \mathbb L}\Gamma=\bigsqcup_{\Gamma, \Gamma'\in \mathbb L}(\Gamma\cap \Gamma'^{t}),\]
where $\mathfrak{\Gamma'}^{-1}=\{w^{-1} ~|~ w \in \mathfrak{\Gamma'}\}$ and $\Gamma'^{t}=\{C^{t}~|~C\in\Gamma'\}$. Also we can associate a unique $\mathfrak{\Gamma} \in \mathfrak{L}$ with each $\Gamma \in \mathbb{L}$, such that $\{w_C^+ ~|~ C \in \Gamma\} \subset \mathfrak{\Gamma}$ by Proposition \ref{prop:d=d2}(2).

Set $A_\mathfrak{\Gamma}:=\{\Gamma' \in \mathbb{L}~|~ \mathfrak{\Gamma'}=\mathfrak{\Gamma}\}$.
For any $w\in W$, let $\mathcal{R}(w)=\{s\in S\:|\:ws<w\}$. 
It is known that $\mathcal{R}(w)=\mathcal{R}(w')$ if $w, w'\in W$ are in the same left cell (see \cite[Lemma 8.6(a)]{Lu03}). So for a left cell $\mathfrak{\Gamma}\in\mathfrak L$, we shall write $\mathcal{R}(\mathfrak{\Gamma})$ instead of $\mathcal{R}(w)$ if $w\in\mathfrak{\Gamma}$.
There is a one-to-one correspondence between the set $A_{\mathfrak{\Gamma}}$  and the set of those $\gamma \in \Lambda_{\ff}$ such that all simple reflections of $W_{\gamma}$ lie in $\mathcal{R}(\mathfrak{\Gamma})$. We shall denote the cardinality of $A_{\mathfrak{\Gamma}}$ by $n_{\mathfrak{\Gamma}}$.

Let $\tilde{\mathbf{J}}$ denote the asymptotic Hecke algebra associated to $\HH$ (cf. \cite{Lu87a}), which admits a $\Z$-basis $\{t_{w}\:|\:w\in W\}$ with multiplication
\[
t_{x}t_{y}=\sum\limits_{z\in W}\gamma_{x, y}^{z^{-1}}t_{z}.
\]

Recall the notation $\mathcal{M}_{\ff}$ introduced in \S\ref{sec:abiform}.
Denote $\mathbf{J}_{\mathfrak{c}}'=\tilde{\mathbf{J}}_{\mathfrak{c}} \otimes_\Z \mathcal{M}_{\ff}$. Let $t_\ff=\sum_{D=(\gamma, d, \gamma) \in \mathcal{D}_{\mathbf{c}}} t_{w_D^+}\otimes e_{\gamma\gamma}$, which is clearly an idempotent element in $\mathbf{J}_{\mathbf{c}}'$. Let $\psi_{\mathbf{c}}: \mathbf{J}_{\mathbf{c}} \rightarrow \mathbf{J}_{\mathfrak{c}}'$ be the $\Z$-map via $\psi_{\mathbf{c}}(t_C)=t_{w_C^+} \otimes e_{\gamma\nu}$ for $C=(\gamma, g, \nu)$.

\begin{lem}\label{mapj}
	\begin{itemize}
		\item[(1)] The map $\psi_{\mathbf{c}}$ preserves the multiplication;
		\item[(2)] As algebras, $\mathbf{J}_{\mathbf{c}} \cong \psi_{\mathbf{c}}(\mathbf{J}_{\mathbf{c}})=t_\ff \mathbf{J}_{\mathfrak{c}}'t_\ff$;
		\item[(3)] It holds that $\mathbf{J}_{\Gamma \cap \Gamma'^t} \cong \psi_{\mathbf{c}}(\mathbf{J}_{\Gamma \cap \Gamma'^t})=\mathbf{J}_{\mathfrak{\Gamma} \cap \mathfrak{\Gamma'}^{-1}} \otimes e_{\gamma'\gamma}$, where $\gamma$ and $\gamma'$ are the common columns of the elements in $\Gamma$ and $\Gamma'$, respectively.
	\end{itemize}
\end{lem}
\begin{proof}
	$(1)$ Recall the homomorphism $\varphi$ from \S\ref{sec:abiform}. For any $A=(\gamma, g, \mu), B=(\mu, g', \nu) \in \mathbf{c}$, we have
	\begin{align*}
		&(\sum_{z \in W} h_{w_A^+,w_B^+}^z \mathcal{C}_z \otimes e_{\gamma\nu})=(\mathcal{C}_{w_A^+} \otimes e_{\gamma\mu})(\mathcal{C}_{w_B^+} \otimes e_{\mu\nu})\\
		&=q^{\ell(w_\circ^\gamma)+\ell(w_\circ^\mu)}\mathfrak{P}_{W_\gamma}\mathfrak{P}_{W_\mu}\varphi(\{A\}\{B\})=q^{\ell(w_\circ^\gamma)+\ell(w_\circ^\mu)}\mathfrak{P}_{W_\gamma}\mathfrak{P}_{W_\mu}\sum_{C \in \Xi} g_{A,B}^C \varphi(\{C\})\\
		&=q^{\ell(w_\circ^\mu)}\mathfrak{P}_{W_\mu}\sum_{C \in \Xi} g_{A,B}^C \mathcal{C}_{w_C^+} \otimes e_{\gamma\nu}.
	\end{align*}
	Therefore, comparing the coefficients, we get $\mathcal{C}_{w_A^+}\mathcal{C}_{w_B^+}=q^{\ell(w_\circ^\mu)}\mathfrak{P}_{W_\mu}\sum_{C \in \Xi} g_{A,B}^C \mathcal{C}_{w_C^+}$. Then by the definition of $\gamma_{A,B}^C$ and Proposition \ref{prop:d=d2}, we have
	$$t_{w_A^+} t_{w_B^+}=\sum_{C \in \Xi, w_C^+ \in \mathfrak{c}} \gamma_{A,B}^{C^t} t_{w_C^+}=\sum_{C \in \mathbf{c}} \gamma_{A,B}^{C^t} t_{w_C^+}.$$
	
	$(2)$ Since $\psi_{\mathbf{c}}(\id_{\mathbf{J}_{\mathbf{c}}})=t_\ff$, $\psi_{\mathbf{c}}(\mathbf{J}_{\mathbf{c}}) \subset t_\ff \mathbf{J}_{\mathfrak{c}}'t_\ff$ and $\psi_{\mathbf{c}}(\mathbf{J}_{\mathbf{c}})=\sum_{C \in \mathbf{c}} \Z (t_{w_C^+} \otimes e_{\gamma\nu})$. Noting that for any $x \in \D_\gamma^+, y \in W$, we have $\mathcal{C}_x\mathcal{C}_y \in \sum_{z \in \D_\gamma^+} \A \mathcal{C}_z$. Then for any $A=(\gamma, g, \mu), B=(\mu', g', \nu) \in \mathbf{c}$, $$\mathcal{C}_{w_A^+}\mathcal{C}_y\mathcal{C}_{w_B^+} \in \sum_{z \in \D_{\gamma\nu}^+} \A \mathcal{C}_z=\sum_{C=(\gamma,g,\nu)} \A \mathcal{C}_{w_C^+}.$$
	Therefore, $t_{w_A^+} t_y t_{w_B^+} \in \sum_{C=(\gamma,g,\nu) \in \mathbf{c}} \Z t_{w_C^+}$. Thus for $D=(\gamma, d, \gamma), D'=(\gamma', d', \gamma') \in \mathcal{D}_\mathbf{c}$, we have $$ (t_{w_{D'}^+}\otimes e_{\gamma'\gamma'}) \mathbf{J}_{\mathfrak{c}}' (t_{w_D^+}\otimes e_{\gamma\gamma}) \subset \sum_{C=(\gamma, g, \gamma') \in \mathbf{c}} \Z (t_{w_C^+} \otimes e_{\gamma'\gamma})=\psi_{\mathbf{c}}(\mathbf{J}_{\mathbf{c}}),$$
	which implies $\psi_{\mathbf{c}}(\mathbf{J}_{\mathbf{c}})=t_\ff \mathbf{J}_{\mathfrak{c}}'t_\ff$.
	
	$(3)$ follows from $(2)$ and the fact that $\mathbf{J}_{\Gamma \cap \Gamma'^t}=t_{D'} \mathbf{J} t_D$, where $D$ and $D'$ are the unique distinguished elements in $\Gamma$ and $\Gamma'$, respectively.
\end{proof}


We assume that $G_{\mathfrak{c}}$ is the finite group associated to $\mathfrak{c}$ in \cite{Lu84}. For each $\mathfrak{\Gamma}\in \mathfrak{L}$, let $H_{\mathfrak{\Gamma}}$ be the associated subgroup of $G_{\mathfrak{c}}$ defined in \cite[\S3.8]{Lu87c}. Let $T$ be the finite $G_{\mathfrak{c}}$-set $\bigoplus_{\mathfrak{\Gamma} \in \mathfrak{L}}G_{\mathfrak{c}}/H_{\mathfrak{\Gamma}}$.
By \cite[Conj. \S3.15]{Lu87c} and \cite[Theorem~4]{BFO09}, the following isomorphisms of rings hold:
\begin{align} \label{eq:BFO}
	\tilde{\mathbf{J}}_{\mathfrak{\Gamma}\cap \mathfrak{\Gamma}^{-1}}\stackrel{\sim}{\longrightarrow} K_{G_{\mathfrak{c}}}(G_{\mathfrak{c}}/H_{\mathfrak{\Gamma}}\times G_{\mathfrak{c}}/H_{\mathfrak{\Gamma}}),
	\qquad
	\tilde{\mathbf{J}}_{\mathfrak{c}}\stackrel{\sim}{\longrightarrow} K_{G_{\mathfrak{c}}}(T\times T).
\end{align}

We set $\bar{T}$ to be the finite $G_{\mathfrak{c}}$-set
\[
\bar{T} = \bigoplus_{\mathfrak{\Gamma} \in \mathfrak{L}}(G_{\mathfrak{c}}/H_{\mathfrak{\Gamma}})^{\oplus{n_{\mathfrak{\Gamma}}}}.
\]
Then by Lemma~\ref{mapj}, we obtain the following counterparts of \eqref{eq:BFO}.
\begin{prop}\label{f4}
	The following isomorphisms of rings hold:
	\[
	\mathbf{J}_{\Gamma \cap \Gamma^t}\stackrel{\sim}{\longrightarrow} K_{G_{\mathfrak{c}}}(G_{\mathfrak{c}}/H_{\mathfrak{\Gamma}}\times G_{\mathfrak{c}}/H_{\mathfrak{\Gamma}}),
	\qquad
	\mathbf{J}_{\mathbf{c}}
	\stackrel{\sim}{\longrightarrow} K_{G_{\mathfrak{c}}}(\bar{T}\times \bar{T}).
	\]
\end{prop}

\subsection{Special modules}
Set $\mathbf{J}_{\mathbf{c}}^+:=\sum_{C \in \mathbf{c}} \mathbb{R}^+ t_C$. A one-dimensional subspace $\mathcal{L}$ in $\mathbf{J}_{\mathbf{c},\C}$ is called a {\em positive line} if $\mathcal{L}^+ :=\mathcal{L} \cap \mathbf{J}_{\mathbf{c}}^+ \neq \emptyset$.
\begin{thm}\label{positive-lines}
	\begin{enumerate}
		\item
		For $\Gamma \in \mathbb L$, there is a unique left ideal $M_{\Gamma}$ of $\mathbf{J}_{\mathbf{c},\C}$ such that the following property $(\heartsuit)$ holds:
		\\
		$(\heartsuit)$ $M_{\Gamma}=\bigoplus_{\Gamma'\in \mathbb L}M_{\Gamma, \Gamma'}$, where $M_{\Gamma, \Gamma'} :=M_{\Gamma}\cap \mathbf{J}_{\Gamma\cap \Gamma'^{t},\C}$ is a positive line for $\forall \Gamma'\in \mathbb L$.
		\item
		Let $\Gamma, \Gamma'\in \mathbb L$ and $A\in \mathbf{c}$. Assume $A\in \tilde{\Gamma}\cap \tilde{\Gamma}'^{t}$ for well-defined $\tilde{\Gamma}, \tilde{\Gamma}'\in \mathbb L$. If $\tilde{\Gamma}\neq \Gamma'$, then $t_{A}M_{\Gamma, \Gamma'}=0$. If $\tilde{\Gamma}=\Gamma'$, then $t_{A}M_{\Gamma, \Gamma'}=M_{\Gamma, \tilde{\Gamma}'}$ and $t_{A}M_{\Gamma, \Gamma'}^{+}=M_{\Gamma, \tilde{\Gamma}'}^{+}$.
		
		\item
		The $\mathbf{J}_{\mathbf{c},\C}$-module $M_{\Gamma}$ is simple. Moreover, $M_{\Gamma} \cong M_{\Gamma'}$ for any $\Gamma, \Gamma'\in\mathbb L$.
		
		\item
		The subspace $\mathbf{I}:=\oplus_{\Gamma\in \mathbb L}M_{\Gamma}\subset\mathbf{J}_{\mathbf{c},\C}$ is a simple two-sided ideal of $\mathbf{J}_{\mathbf{c},\C}$.
		
		\item
		Take $\Gamma, \Gamma', \tilde{\Gamma}, \tilde{\Gamma}'\in\mathbb L$. If $\Gamma\neq \tilde{\Gamma}'$, then $M_{\Gamma, \Gamma'}M_{\tilde{\Gamma}, \tilde{\Gamma}'}=0$. If $\Gamma=\tilde{\Gamma}'$, then the multiplication in $\mathbf{J}_{\mathbf{c},\C}$ defines an isomorphism $M_{\Gamma, \Gamma'}\otimes M_{\tilde{\Gamma}, \Gamma}\stackrel{\sim}{\longrightarrow} M_{\tilde{\Gamma}, \Gamma'}$ and a surjective map $M_{\Gamma, \Gamma'}^{+}\times M_{\tilde{\Gamma}, \Gamma}^{+}\twoheadrightarrow M_{\tilde{\Gamma}, \Gamma'}^{+}$.
		
		\item
		The anti-involution $\Psi_{\mathbf{c}} :\mathbf{J}_{\mathbf{c},\C} \rightarrow \mathbf{J}_{\mathbf{c},\C}$ given by $t_{C}\mapsto t_{C^{t}}\ (\forall C\in \mathbf{c})$ maps $M_{\Gamma, \Gamma'}$ (resp. $M_{\Gamma, \Gamma'}^{+}$) onto $M_{\Gamma', \Gamma}$ (resp. $M_{\Gamma', \Gamma}^{+}$) for $\Gamma, \Gamma'\in\mathbb L$.
	\end{enumerate}
\end{thm}

The simple $\mathbf{J}_{\mathbf{c},\C}$-modules $M_{\Gamma}$ defined in Theorem~ \ref{positive-lines} are called the {\em special $\mathbf{J}_{\mathbf{c},\C}$-modules}.

\begin{proof}
	By \cite[\S1.2]{Lu18}, we have a positive line $M_{\mathfrak{\Gamma},\mathfrak{\Gamma'}}$ in $\tilde{\mathbf{J}}_{\mathfrak{\Gamma} \cap \mathfrak{\Gamma'}^{-1},\C}$ stable under the action of $\tilde{\mathbf{J}}_{\mathfrak{\Gamma'}\cap\mathfrak{\Gamma'}^{-1},\C}$. Considering the action of $\sum_{z \in \mathfrak{\Gamma} \cap \mathfrak{\Gamma'}^{-1}} t_z$ on $\tilde{\mathbf{J}}_{\mathfrak{\Gamma} \cap \mathfrak{\Gamma'}^{-1},\C}$ and applying Perron's theorem \cite{Pe07}, we can see that $M_{\mathfrak{\Gamma},\mathfrak{\Gamma'}}$ is the unique one in $\tilde{\mathbf{J}}_{\mathfrak{\Gamma} \cap \mathfrak{\Gamma'}^{-1},\C}$ stable under the action of
	$\sum_{z \in \mathfrak{\Gamma} \cap \mathfrak{\Gamma'}^{-1}} t_z$.
	
	For $\Gamma, \Gamma' \in \mathbb{L}$, let $\gamma$ and $\gamma'$ be the common columns of the elements in $\Gamma$ and $\Gamma'$, respectively. Set $M_{\Gamma,\Gamma'}=M_{\mathfrak{\Gamma},\mathfrak{\Gamma}'} \otimes e_{\gamma'\gamma} \subset \mathbf{J}_{\mathfrak{c},\C}'$.
	Using $\psi_{\mathbf{c}}$, we regard $\mathbf{J}_{\mathbf{c}, \C}$ as a subspace of $\mathbf{J}_{\mathfrak{c}, \C}'$, then $M_{\Gamma,\Gamma'}$ is the unique positive line in $\mathbf{J}_{\Gamma \cap \Gamma'^t,\C}$ stable under the action of $\mathbf{J}_{\Gamma' \cap \Gamma'^t,\C}$. Set $M_\Gamma=\bigoplus_{\Gamma' \in \mathbb{L}} M_{\Gamma,\Gamma'}$. It can be checked that $M_\Gamma$ is a left ideal of $\mathbf{J}_{\mathbf{c}, \C}$ and satisfies the propterty $(\heartsuit)$. The uniqueness of $M_\Gamma$ follows from the uniqueness of $M_{\Gamma,\Gamma'}$.
	
	Using \cite[Theorem~1.2 $(b)$-$(f)$]{Lu18}, we can prove the properties $(2)$-$(6)$ easily so we will only prove $(2)$ and leave the others to the readers. Recall that $\gamma$, $\gamma'$, $\tilde{\gamma}$ and $\tilde{\gamma}'$ are the common columns of the elements in $\Gamma$, $\Gamma'$, $\tilde{\Gamma}$ and $\tilde{\Gamma}'$, respectively. Note that $\tilde{\Gamma}=\Gamma'$ iff $\tilde{\gamma}=\gamma'$ and $\tilde{\mathfrak{\Gamma}}=\mathfrak{\Gamma'}$. So $t_A M_{\Gamma,\Gamma'}=(t_{w_A^+} \otimes e_{\tilde{\gamma}'\tilde{\gamma}}) (M_{\mathfrak{\Gamma},\mathfrak{\Gamma'}} \otimes e_{\gamma'\gamma})$. Then by \cite[Theorem~1.2 $(b)$]{Lu18},
	\begin{align*}
		t_A M_{\Gamma,\Gamma'}=(t_{w_A^+} \otimes e_{\tilde{\gamma}'\tilde{\gamma}}) (M_{\mathfrak{\Gamma},\mathfrak{\Gamma'}} \otimes e_{\gamma'\gamma})&=\delta_{\tilde{\gamma}\gamma'} \delta_{\tilde{\mathfrak{\Gamma}}\mathfrak{\Gamma'}}M_{\mathfrak{\Gamma},\tilde{\mathfrak{\Gamma'}}} \otimes e_{\tilde{\gamma}'\gamma}=\delta_{\tilde{\Gamma}\Gamma'} M_{\Gamma,\Gamma'}\\
		t_A M_{\Gamma,\Gamma'}^+=(t_{w_A^+} \otimes e_{\tilde{\gamma}'\tilde{\gamma}}) (M_{\mathfrak{\Gamma},\mathfrak{\Gamma'}}^+ \otimes e_{\gamma'\gamma})&=\delta_{\tilde{\gamma}\gamma'} \delta_{\tilde{\mathfrak{\Gamma}}\mathfrak{\Gamma'}} M_{\mathfrak{\Gamma},\tilde{\mathfrak{\Gamma'}}}^+ \otimes e_{\tilde{\gamma}'\gamma}=\delta_{\tilde{\Gamma}\Gamma'} M_{\Gamma,\Gamma'}^+.
	\end{align*}
\end{proof}

In \cite[\S1.7]{Lu18}, we can see that $M_{\mathfrak{\Gamma},\mathfrak{\Gamma'}}$ can be defined over $\Q$. That is, $M_{\mathfrak{\Gamma},\mathfrak{\Gamma'}} \cap \tilde{\mathbf{J}}_{\mathfrak{c}, \Q}$ spans $M_{\mathfrak{\Gamma},\mathfrak{\Gamma'}}$. Therefore, $M_{\Gamma, \Gamma'}$ can also be defined over $\Q$ by the setting of $M_{\Gamma, \Gamma'}$.

\subsection{Blocks of $\Sc_{q}$}
In this subsection, we shall follow \cite{Ro99} to give a characterization of the blocks in the $q$-Schur algebra $\Sc_{q}$. Let $\mathfrak{O}=\{f/g\in \K\:|\:f \in \A, g \in 1+q\Z[q]\}$.

\begin{lem}\label{g2}
Let $\mathbf{c}\in B_{0}$ and assume that $D\in \mathcal{D}_{\mathbf{c}}$. Then $t_{D}$ is a primitive idempotent in $\mathbf{J}_{\mathfrak{O}}$. As a result, $\{t_{\mathbf{c}}\:|\:\mathbf{c}\in B_{0}\}$ is the complete set of primitive central idempotents of $\mathbf{J}_{\mathfrak{O}}$.
\end{lem}

\begin{proof}
	Recall from \S\ref{sec:asymp} that we have the $\Z$-map $\tau$ on $\mathbf{J}$ via $\tau(t_{C})=1$ if $C\in \mathcal{D}$ and $\tau(t_{C})=0$ otherwise which can induce a non-degenerate symmetric associative bilinear form. From \cite[\S1.3 (d)]{Lu87c}, we have a decomposition $\tau_\Q=\sum_{\lambda \in \mathcal{E}} c_\lambda \mathrm{Tr}(\cdot, E_{\spadesuit}^\lambda)$ where $c_\lambda$ is a positive rational number. Let $e$ be an idempotent of $\mathbf{J}_\mathfrak{O}$, such that $e t_D= t_D e=e$ for some $D \in \mathcal{D}$, then $\mathrm{Tr}(e, E_{\spadesuit, \K}^\lambda)$ is an positive integer not more than $\mathrm{Tr}(t_D, E_{\spadesuit, \K}^\lambda)=1$, therefore, $\tau_\mathfrak{O}(e) \in \mathfrak{O} \cap \Q=\Z$ is an positive integer not more than $\tau(t_D)=1$. So we can deduce that either $\tau_\mathfrak{O}(e)=0$ and $\mathrm{Tr}(e, E_{\spadesuit, \K}^\lambda)=0$ for all $\lambda$, or $\tau_\mathfrak{O}(e)=1$ and $\mathrm{Tr}(e, E_{\spadesuit, \K}^\lambda)=\mathrm{Tr}(t_D, E_{\spadesuit, \K}^\lambda)$ for all $\lambda$. Thus $e=0$ or $e=t_D$.
\end{proof}

By Theorem~\ref{thm:cell structures}, we have
\begin{align*}
	q^{\mathfrak{a}(C)-\ell(w_\circ^{\co(C)})}\Phi (\{C\}) \in t_{C}+\sum_{A \in \Xi, \mathfrak{a}(A)=\mathfrak{a}(C)} q\Z[q] t_{A}+\sum_{B \in \Xi, \mathfrak{a}(B)> \mathfrak{a}(C)} \A t_{B} .
\end{align*}
Therefore, there exists some $m \in \Z$ such that $q^{m}\mathrm{det}(\Phi)\in 1+q\Z[q]$, and hence $\mathrm{det}(\Phi)$ is invertible in $\mathfrak{O}$. So $\Phi_{\mathfrak{O}}$ is an isomorphism.

From Lemma \ref{g2}, we immediately obtain the following result (cf. \cite[Theorem 1]{Ro99} or \cite[Theorem 4.3]{Ge05}).

\begin{thm}\label{g4}
The following two conditions are equivalent:
\begin{itemize}
  \item[(a)] $\lambda, \lambda'\in\mathcal{E}$ belong to the same two-sided cell. That is, there exists some $\mathbf{c}\in B_{0}$ such that $t_{\mathbf{c}}E_{\spadesuit}^{\lambda}\neq 0$ and $t_{\mathbf{c}}E_{\spadesuit}^{\lambda'}\neq 0$;
  \item[(b)] $\lambda, \lambda'\in \mathcal{E}$ belong to the same block of $\Sc_{q, \mathfrak{O}}$. That is, there exists a primitive central idempotent $e$ of $\Sc_{q, \mathfrak{O}}$ such that $eE_{q}^{\lambda}\neq 0$ and $eE_{q}^{\lambda'}\neq 0$.
\end{itemize}
\end{thm}

\subsection{Left cell representations}
In this subsection, we will follow \cite[\S21]{Lu03} to study the left cell representations of $\mathcal{S}_{\Q}$. Let $\Gamma$ be a left cell in $\Xi$, then we have $\mathbf{J}_{\Gamma}=\mathbf{J}t_D$ where $D$ is the unique distinguished element in $\Gamma$.

We denote by $[\Gamma]$ the $\Sc_{\Q}$-module induced by $\Phi_{\Q}$ from $\mathbf{J}_{\Gamma,\Q}$. We call $[\Gamma]$ a {\em left cell representation} of $\mathcal{S}_{\Q}$. We have $[E^{\lambda} : [\Gamma]]=[E_{\spadesuit}^{\lambda} : \mathbf{J}_{\Gamma,\Q}]$ for any simple $\mathcal{S}_{\Q}$-module $E^{\lambda}$, where $[E^{\lambda} : [\Gamma]]$ (resp. $[E_{\spadesuit}^{\lambda} : \mathbf{J}_{\Gamma,\Q}]$) denotes the multiplicity of $E^{\lambda}$ (resp. $E_{\spadesuit}^{\lambda}$) as a simple component of $[\Gamma]$ (resp. $\mathbf{J}_{\Gamma,\Q}$).

Following from basic representation theory of algebras, we have the following results.
\begin{lem}\label{g6}
The $\Q$-linear map $$u :\mathrm{Hom}_{\mathbf{J}_{\Q}}(\mathbf{J}_{\Gamma,\Q}, E_{\spadesuit}^{\lambda})\rightarrow t_{D}E_{\spadesuit}^{\lambda},\qquad\xi\mapsto \xi(t_{D})$$ is an isomorphism, and $\Tr(t_{D}, E_{\spadesuit}^{\lambda})=[E^{\lambda} : [\Gamma]]$.
\end{lem}

\begin{lem}\label{h4}
Let $\Gamma, \Gamma'$ be two left cells in $\Xi$.
\begin{itemize}
\item[(1)] We have $\mathrm{Hom}_{\mathbf{J}_{\Q}}(\mathbf{J}_{\Gamma,\Q}, \mathbf{J}_{\Gamma',\Q})=0$ unless $\Gamma, \Gamma'$ are contained in the same two-sided cell.
\item[(2)] In general, we have $\mathrm{dim} \mathrm{Hom}_{\mathbf{J}_{\Q}}(\mathbf{J}_{\Gamma,\Q}, \mathbf{J}_{\Gamma',\Q})=\sharp(\Gamma\cap \Gamma'^{t})$; in particular, $\Gamma\cap \Gamma'^{t}=\emptyset$ unless $\Gamma, \Gamma'$ are contained in the same two-sided cell.
\item[(3)] The subspace $\mathbf{J}_{\Gamma\cap \Gamma^{t},\Q}$ of $\mathbf{J}_{\Q}$ is a $\Q$-algebra with a unit element $t_{D}$. Moreover, it is isomorphic to $\mathrm{End}_{\mathbf{J}_{\Q}}(\mathbf{J}_{\Gamma,\Q})$, and therefore, it is split semisimple.
\end{itemize}
\end{lem}

In fact, we can also imitate \cite[\S21.9]{Lu03} to show that the $\Q$-subalgebra $\mathbf{J}_{\Gamma\cap \Gamma^{t},\Q}$ is split semisimple with a pair of dual bases $\{t_{A}\:|\:A\in \Gamma\cap \Gamma^{t}\}$ and $\{t_{A^{t}}\:|\:A\in \Gamma\cap \Gamma^{t}\}$ with respect to some bilinear form.


Recall from Corollary \ref{c6} that $f_\lambda$ is a Schur element of $\mathbf{J}_{\Q}$.
\begin{prop}\label{g8}
Let $\lambda, \lambda' \in \mathcal{E}$. Then we have
$$\sum_{A\in \Gamma\cap \Gamma^{t}}\Tr(t_{A}, E_{\spadesuit}^{\lambda}) \Tr(t_{A^{t}}, E_{\spadesuit}^{\lambda'})=\delta_{\lambda \lambda'} f_\lambda[E^{\lambda} : [\Gamma]].$$
\end{prop}
\begin{proof}
	Note that $t_D E_{\spadesuit}^\lambda$ is a $\mathbf{J}_{\Gamma \cap \Gamma^{t},\Q}$-module and $$\Tr(t_{A}, E_{\spadesuit}^\lambda)=\Tr(t_{A} t_{D}, E_{\spadesuit}^\lambda)=\Tr(t_{A}, t_{D} E_{\spadesuit}^\lambda)\quad\mbox{for any $A \in \Gamma \cap \Gamma^{t}$}.$$
Thanks to Lemma \ref{g6}, we have
	$$\sum_{A\in \Gamma\cap \Gamma^{t}}\Tr(t_{A}, E_{\spadesuit}^{\lambda}) \Tr(t_{A^{t}}, E_{\spadesuit}^{\lambda'})=\delta_{\lambda \lambda'} f_\lambda \dim t_{D} E_{\spadesuit}^\lambda=\delta_{\lambda \lambda'} f_\lambda[E^{\lambda} : [\Gamma]].$$
\end{proof}

Fix a simple $\mathbf{J}_{\Q}$-module $E_{\spadesuit}^{\lambda}$ and choose a $\Q$-basis of $E_{\spadesuit}^{\lambda}$. Recall the matrix representation $\rho^\lambda:\mathbf{J}_{\Q} \rightarrow \Mat_{d_\lambda}(\Q)$ introduced in \S\ref{subsec:rep}. We set $$\chi_{A}^{\lambda} :=\Tr(t_{A}, E_{\spadesuit}^{\lambda})=\sum_{\mathfrak{s} \in M(\lambda)} \rho_{\mathfrak{ss}}^\lambda(t_A)\quad \mbox{for $A\in\Xi$}.$$
By an argument similar to \cite[Lemma 20.13(a)]{Lu03}, we see that $\chi_{A^{t}}^{\lambda}$ is the complex conjugate of $\chi_{A}^{\lambda}$. Then we have the following result (cf. \cite[\S3.5]{Lu87c} or \cite[Remark 3.5]{Ge14}).
\begin{lem}\label{h5}
For $A\in \Xi$, $\chi_{A}^{\lambda}\neq 0$ for some $\lambda\in \mathcal{E}$ if and only if $A\sim_{L} A^{t}$.
\end{lem}

The following proposition is inspired by \cite[Corollary 5.8]{Lu84} and also \cite[Theorem 1.8.1]{GJ11}.
\begin{prop}\label{h6}
Let $\Gamma$ be a left cell in $\Xi$ and $\lambda, \lambda' \in \mathcal{E}$.
\begin{enumerate}
\item
If $\lambda \neq \lambda'$, then for any $\mathfrak{s,t} \in M(\lambda)$ and $\mathfrak{u, v}\in M(\lambda')$, we have
\begin{align*}
\sum_{A\in \Gamma}\rho^\lambda_{\mathfrak{st}}(t_A)\rho^{\lambda'}_{\mathfrak{uv}}(t_{A^t})=0\quad \text{and}\quad \sum_{A\in \Gamma}\chi_{A}^{\lambda} \chi_{A^{t}}^{\lambda'}=0.
\end{align*}
\item
If $\lambda=\lambda'$, then for any $\mathfrak{s} \in M(\lambda)$ we have
\begin{align*}
\sum\limits_{\mathfrak{t} \in M(\lambda)}\sum_{A\in \Gamma}\rho^\lambda_{\mathfrak{st}}(t_A)\rho^\lambda_{\mathfrak{ts}}(t_{A^t})=\sum_{A\in \Gamma}\chi_{A}^{\lambda} \chi_{A^{t}}^{\lambda}=[E_{\spadesuit}^{\lambda} : \mathbf{J}_{\Gamma,\Q}] f_\lambda.
\end{align*}
 In particular, $[E_{\spadesuit}^{\lambda} : \mathbf{J}_{\Gamma,\Q}]> 0$ if and only if there exists at least one $A\in \Gamma$ such that $\chi_{A}^{\lambda}\neq 0$.
\end{enumerate}
\end{prop}

Recall from \S\ref{sec:asymp} the $\Z$-map $\tau$ on $\mathbf{J}$.
We have the following result (cf. \cite[Proposition 3.6]{Ge11}).
\begin{prop}\label{i2}
For any $A, B, C\in \Xi$, we have
\begin{align*}
\gamma_{A, B}^{C}=\sum_{\lambda\in \mathcal{E}}\sum_{\mathfrak{s,t,u} \in M(\lambda)}\frac{1}{f_\lambda} \rho^\lambda_{\mathfrak{st}}(t_A)\rho^\lambda_{\mathfrak{tu}}(t_B)\rho^\lambda_{\mathfrak{us}}(t_C).
\end{align*}
\end{prop}
\begin{proof}
By \cite[Proposition 19.2(c)]{Lu03}, it is a direct calculation that
\begin{align*}
\gamma_{A, B}^{C}&=\tau(t_{A}t_{B}t_{C})=\sum_{\lambda\in \mathcal{E}} \frac{1}{f_\lambda} \Tr(t_{A}t_{B}t_{C}, E_{\spadesuit}^{\lambda})\\
&=\sum_{\lambda\in \mathcal{E}} \frac{1}{f_\lambda} \Tr\big(\rho^\lambda(t_{A}t_{B}t_{C})\big)=\sum_{\lambda\in \mathcal{E}} \frac{1}{f_\lambda} \Tr\big(\rho^\lambda(t_{A})\rho^\lambda(t_{B})\rho^\lambda(t_{C})\big)\\
&=\sum_{\lambda\in \mathcal{E}}\sum_{\mathfrak{s,t,u} \in M(\lambda)}\frac{1}{f_\lambda} \rho^\lambda_{\mathfrak{st}}(t_A)\rho^\lambda_{\mathfrak{tu}}(t_B)\rho^\lambda_{\mathfrak{us}}(t_C).
\end{align*}
\end{proof}

\begin{rem}
Inspired by Proposition \ref{i2}, we can imitate \cite{Ge09} to construct $\mathbf{J}_{\Q}$ in an entirely different fashion, without any reference to the canonical basis $\big\{\{A\}\:|\:A\in \Xi\big\}$. We omit the details here.
\end{rem}

The following result gives an extremely strong condition on the expansion of a left cell representation in terms of the irreducible ones (cf. \cite[Lemma 4.6]{Ge05}).
\begin{lem}\label{h1}
Let $\Gamma$ be a left cell in $\Xi$. Then we have
\begin{align*}
1=\sum_{\lambda\in\mathcal{E}}\frac{1}{f_\lambda}[E^{\lambda} : [\Gamma]].
\end{align*}
\end{lem}
\begin{proof}
By \cite[Proposition 19.2(c)]{Lu03}, we have
\begin{align*}
1=\tau(t_{D})=\sum_{\lambda\in\mathcal{E}}\frac{1}{f_\lambda}\Tr(t_{D}, E_{\spadesuit}^{\lambda}).
\end{align*}
On the other hand, by Lemma \ref{g6}, we have $\Tr(t_{D}, E_{\spadesuit}^{\lambda})=[E^{\lambda} : [\Gamma]]$ for each $\lambda\in\mathcal{E}$. We then obtain the desired result.
\end{proof}

From Lemma \ref{h1}, we can easily obtain the following result.
\begin{cor}\label{h2}
Let $\Gamma$ be a left cell in $\Xi$. Assume that there exists some $\lambda\in\mathcal{E}$ such that $[E^{\lambda} : [\Gamma]]\neq 0$ and $f_\lambda=1$. Then $[\Gamma]\cong E^{\lambda}$.
\end{cor}

\begin{cor}\label{h3}
Let $\Gamma$ be a left cell in $\Xi$. Assume that $f_\lambda=1$ for all $\lambda\in\mathcal{E}$. Then $[\Gamma]$ is irreducible.
\end{cor}

\subsection{Application: cellular bases}
Finally, let us look at an application of Corollary \ref{h3}, which is similar to the discussions in \cite[Example 4.2]{Ge07}.

Assume that there are no bad primes, i.e. $f_{\lambda}=1$ for all $\lambda\in \mathcal{E}$. Then we claim that there are signs $\delta_{A}=\pm 1$ ($A\in \Xi$) such that
\begin{align*}
\big\{\delta_{A}\{A\}\:|\:A\in \Xi\big\} \text{ is a cellular basis of $\Sc_{q}$.}
\end{align*}

Let $\Gamma$ be a left cell in $\Xi$. Then $[\Gamma]$ is a simple $\mathcal{S}_{\Q}$-module by Corollary \ref{h3}. Moreover, it is easy to see that all simple $\mathcal{S}_{\Q}$-modules arise in this way. Thus, for each $\lambda\in \mathcal{E}$, we can choose a unique left cell $\Gamma^{\lambda}$ such that $\mathbf{J}_{\Gamma^{\lambda}, \Q} \cong E_{\spadesuit}^{\lambda}$. Let us write
\begin{align*}
\Gamma^{\lambda}=\{E_{\mathfrak{s}}\:|\:\mathfrak{s}\in M(\lambda)\} \quad \text{ for }\lambda\in \mathcal{E}.
\end{align*}
Then $\{t_{E_{\mathfrak{s}}}\:|\:\mathfrak{s}\in M(\lambda)\}$ is a basis of $\mathbf{J}_{\Gamma^{\lambda},\Q}$ and the corresponding matrix coefficients are
\begin{align*}
\rho_\mathfrak{st}^\lambda(t_A)=\gamma_{A, E_{\mathfrak{t}}}^{E_{\mathfrak{s}}^{t}}\quad \text{ for $A\in \Xi$ and $\mathfrak{s,t} \in M(\lambda)$.}
\end{align*}
Since $\gamma_{A, E_{\mathfrak{t}}}^{E_{\mathfrak{s}}^{t}}=\gamma_{E_{\mathfrak{t}}^{t}, A^{t}}^{E_{\mathfrak{s}}}=\gamma_{A^{t}, E_{\mathfrak{s}}}^{E_{\mathfrak{t}}^{t}}$, we have $\rho^\lambda(t_{A^t})=\rho^\lambda(t_A)^{\mathrm{tr}}$ for all $A\in \Xi$. Thus, if we take $P^\lambda$ to be the identity matrix of size $d_\lambda$, then the conditions $(a)$ and $(b)$ in Proposition \ref{c4} are satisfied. Moreover, for fixed $\mathfrak{s,t} \in M(\lambda)$, the Schur relation \eqref{c2} becomes
\begin{align*}
\sum_{A\in \Xi}(\gamma_{A, E_{\mathfrak{t}}}^{E_{\mathfrak{s}}^{t}})^{2}=\sum_{A\in \Xi} \rho_\mathfrak{st}^\lambda(t_A) \rho_\mathfrak{ts}^{\lambda}(t_{A^t})=f_\lambda=1.
\end{align*}
Therefore, we deduce that there is a unique $C=C_{\lambda}(\mathfrak{s}, \mathfrak{t})\in \Xi$ such that $\gamma_{C, E_{\mathfrak{t}}}^{E_{\mathfrak{s}}^{t}}\neq 0$, and moreover, $\gamma_{A, E_{\mathfrak{t}}}^{E_{\mathfrak{s}}^{t}}=\pm 1$. Now the formula in \eqref{cst} reads
\begin{align*}
C_\mathfrak{s,t}^\lambda=\sum_{A\in \Xi}\gamma_{A^{t}, E_{\mathfrak{s}}}^{E_{\mathfrak{t}}^{t}}\{A\}=\sum_{A\in \Xi}\gamma_{A, E_{\mathfrak{t}}}^{E_{\mathfrak{s}}^{t}}\{A\}=\gamma_{C, E_{\mathfrak{t}}}^{E_{\mathfrak{s}}^{t}}\{C\}=\delta_{C}\{C\},
\end{align*}
where $C=C_{\lambda}(\mathfrak{s}, \mathfrak{t})$ and $\delta_{C}=\pm 1$. Thus, $\big\{\delta_{A}\{A\}\:|\:A\in \Xi\big\}$ is a cellular basis of $\Sc_{q}$ as claimed.

The above assumptions are satisfied for $W$ of type $A$. Thanks to a geometric interpretation of the canonical basis $\{A\}$, we know that $g_{A,B}^{C}\in \N[q,q^{-1}]$ by \cite[Theorem 4.5]{LW22}. Hence we have $\gamma_{A,B}^{C}\in \N$ for all $A, B, C$. So, in this case, $\big\{\{A\}\:|\:A\in \Xi\big\}$ is a cellular basis of $\Sc_{q}$.


\end{document}